\documentclass{amsart}
\usepackage{hyperref}
\usepackage{cite,
}
\usepackage{amssymb}
\usepackage{amsmath}
\usepackage{amsthm}
\usepackage{amsfonts}
\usepackage{graphicx}
\usepackage{amscd}
\usepackage{mathrsfs}
\usepackage{euscript}

\allowdisplaybreaks
\makeatletter
\@ifundefined{pdfoutput}
{
\DeclareGraphicsRule{.wmf}{bmp}{}{}
\DeclareGraphicsRule{.jpg}{bmp}{}{}
\DeclareGraphicsRule{.png}{bmp}{}{}
\DeclareGraphicsRule{.cdr}{bmp}{}{}
\DeclareGraphicsRule{.gif}{bmp}{}{}
}
{
\ifnum\pdfoutput=0\relax
\DeclareGraphicsRule{.wmf}{bmp}{}{}
\DeclareGraphicsRule{.jpg}{bmp}{}{}
\DeclareGraphicsRule{.png}{bmp}{}{}
\DeclareGraphicsRule{.cdr}{bmp}{}{}
\DeclareGraphicsRule{.gif}{bmp}{}{}
\fi
\ifnum\pdfoutput=1\relax
\DeclareGraphicsRule{.eps}{pdf}{}{}
\DeclareGraphicsRule{.wmf}{jpg}{}{}
\fi
}
\makeatother
\newtheorem{thm}{Theorem}[section]
\newtheorem{cor}[thm]{Corollary}

\newtheorem{lem}[thm]{Lemma}
\newtheorem{notation}[thm]{Notation}
\newtheorem{prop}[thm]{Proposition}

\theoremstyle{definition}

\newtheorem{defn}[thm]{Definition}

\newtheorem{example}{Example}[section]
\newtheorem{remark}[thm]{Remark}
\numberwithin{equation}{section}

\title[A subelliptic Taylor isomorphism on infinite Heisenberg groups]
	{A subelliptic Taylor isomorphism on infinite-dimensional Heisenberg groups}
\author[Gordina]{Maria Gordina{$^{*}$}}
\address{Department of Mathematics\\
University of Connecticut\\
Storrs, CT 06269, USA }
\email{maria.gordina@uconn.edu}
\thanks{\footnotemark {$^{*}$} This research was supported in part by NSF
Grant  DMS-1007496, Michler fellowship}
\author[Melcher]{Tai Melcher{$^{\dagger }$}}
\thanks{\footnotemark {$^\dagger$} This research was supported in part by NSF
Grant DMS-0907293}
\address{Department of Mathematics\\
University of Virginia\\ Charlottesville, VA 22903 USA}
\email{melcher@virginia.edu}


\begin{document}
\begin{abstract}
Let $G$ denote an infinite-dimensional Heisenberg-like group, which
is a class of infinite-dimensional step 2 stratified Lie groups.
We consider holomorphic functions on $G$ that are square integrable
with respect to a heat kernel measure which is formally subelliptic,
in the sense that all appropriate finite-dimensional projections
are smooth measures.  We prove a unitary equivalence between a
subclass of these square integrable holomorphic functions and a
certain completion of the universal enveloping algebra of the
``Cameron-Martin'' Lie subalgebra.  The isomorphism defining the
equivalence is given as a composition of restriction and Taylor
maps.  \end{abstract}

\maketitle
\tableofcontents

\section{Introduction\label{s.intro}}

We study spaces of holomorphic functions on infinite-dimensional
Heisenberg-like groups based on an abstract Wiener space as constructed
in \cite{DG08-2}.  In particular, we consider holomorphic functions
which are square integrable with respect to a subelliptic heat
kernel measure and prove a unitary equivalence between a subclass
of these functions and a certain completion of the universal
enveloping algebra of the Cameron-Martin Lie subalgebra.  These
results may be viewed as an analogue of the results in \cite{DG08-3}
for degenerate heat kernel measures, or as an extension of the
finite-dimensional results in \cite{DGSC09-2} to a special
infinite-dimensional case.  Perhaps more particularly, it is an
infinite-dimensional extension of \cite{DGSC09-1} in a special case,
as the Heisenberg-like groups considered here are nilpotent. There
are considerable differences from both cases in techniques, as
analytically our setting is very different from the elliptic case
in \cite{DG08-3}, and there are numerous subtle issues when dealing
with infinite dimensions versus the finite-dimensional nilpotent
case in \cite{DGSC09-1}.  In particular, in the infinite-dimensional
setting, it is necessary to consider two different norms on the Lie
algebra, one which defines the space on which the functions live
and one which controls the analysis.  This is directly analogous
to the abstract Wiener space construction.

\subsection{Background}
We give a brief (incomplete) background of the development of the Taylor
isomorphism to put our results into context.  See the papers cited
here and their bibliographies for more complete references.  Also,
the paper \cite{GM96} gives a very nice discussion and extensive
history of the theory.

Let us first recall the classical result.
Let $f:\mathbb{C}\rightarrow\mathbb{C}$ be a holomorphic function.  Then it is
well known that $f$ is everywhere determined by the values of its
derivatives at the origin and in particular
\[ f(z) = \sum_{k=0}^\infty \frac{f^{(k)}(0)}{k!}z^k. \]
Moreover, if $d\mu_t(z) = p_t(z)\,dz$ where $p_t(z) = \frac{1}{\pi t}
e^{-|z|^2/t}$
is the standard Gaussian density on $\mathbb{C}$, then $\langle z^k,
z^\ell\rangle_{L^2(\mu_t)}=\delta_{k\ell}t^k k!$, which implies that
\begin{equation}
\label{e.c}
\|f\|_{L^2(\mu_t)}^2 = \sum_{k=0}^\infty \frac{t^k}{k!}|f^{(k)}(0)|^2.
\end{equation}
Thus, one may consider the Taylor expansion as an isometric isomorphism
from the space of square integrable, holomorphic functions
onto the sequence space of derivatives at 0 endowed with an appropriate norm.

This isomorphism first appeared in the paper of Fock \cite{Fock28}
(actually for $\mathbb{C}^n$),
but was not made explicit until the work of Segal \cite{Segal56,Segal62}
and Bargmann \cite{Bargmann61}.  Multiple authors contributed to various extensions
of this theory, all of which culminated in the paper \cite{DG97}.  In this
paper, Driver and Gross considered the case of a connected complex
(finite-dimensional) Lie group $G$ with Lie algebra $\mathfrak{g}$.
Equip $\mathfrak{g}$ with any inner product, and suppose that
$\{V_i\}_{i=1}^n$ is an orthonormal basis of $\mathfrak{g}$.  Consider
$L=\sum_{i=1}^n \tilde{V}_i^2$, where $\tilde{V}$ is the left
invariant vector on $G$ field associated to $V\in\mathfrak{g}$.
Then $L$ is an elliptic second order differential operator, and we
let $\{g_t\}_{t\ge0}$ denote a Brownian motion on $G$ with generator
$L$.  For $t>0$, let $\mathcal{H}L^2(G,\mu_t)$ denote the space of
holomorphic functions on $G$ which are square integrable with respect
to the heat kernel measure $\mu_t=\mathrm{Law}(g_t)$ on $G$.  Then
it was proved in \cite{DG97} that the analogous Taylor map in this setting
is an isometric isomorphism from $\mathcal{H}L^2(G,\mu_t)$ to the space
of derivatives at the identity equipped with a norm inspired by the
expression in (\ref{e.c}).

Recently, in \cite{DGSC09-2}, Driver, Gross, and
Saloff-Coste have further extended this theory to the case of subelliptic
(or hypoelliptic) heat kernel measures on a connected complex Lie group.  That is, suppose in
the previous setting that
$\{V_i\}_{i=1}^k\subset\mathfrak{g}$ is not itself a full basis of
$\mathfrak{g}$, but does satisfy the
H\"ormander (or bracket generating) condition
\[ \mathrm{span}\{V_i,[V_i,V_j],[V_i,[V_j,V_k]],\ldots\} = \mathfrak{g}. \]
Then due to the classical result of H\"ormander \cite{Hormander67}, it is
well known that, for the process $\{g_t\}_{t\ge0}$ generated by $L=\sum_{i=1}^k
\tilde{V}_i^2$, $\mu_t=\mathrm{Law}(g_t)$ is a smooth measure for all $t>0$.  In
\cite{DGSC09-2}, it is proved that the Taylor map is an isometric
isomorphism, this time from $\mathcal{H}L^2(G,\mu_t)$ onto the space of
derivatives at the identity with an appropriately modified norm.

There have also been several infinite-dimensional settings in which
Taylor isomorphisms have been shown to hold.  In particular, in
\cite{DG08-3} Driver and the first named author proved a Taylor
isomorphism theorem for nondegenerate heat kernel measure on the
same infinite-dimensional Heisenberg-like groups considered in the
present paper.  The first named author has proved analogues on the
infinite-dimensional complex Hilbert-Schmidt groups
\cite{Gordina00-2,Gordina00-1} and for the group of invertible
operators in a factor of type II$_1$ \cite{Gordina02}.  Also, in
\cite{Cecil08}, Cecil proved an analogue for path groups over stratified
nilpotent Lie groups.  To our knowledge, the present paper represents
the first analogous result for an infinite-dimensional subelliptic
setting.

\subsection{Statement of Results}
\subsubsection{Heisenberg-like groups and subelliptic heat kernel measures}
Let $(W,H,\mu)$ be a complex abstract Wiener space and let $\mathbf{C}$
be a finite-dimensional complex inner product space.  Let
$\mathfrak{g}=W\times \mathbf{C}$ be an infinite-dimensional
Heisenberg-like Lie algebra, which is constructed as an infinite-dimensional step 2 nilpotent Lie algebra with Lie bracket satisfying
the following condition:
\begin{equation}
\label{e.h}
[W,W]=\mathbf{C}.
\end{equation}
Let $G$ denote $W\times\mathbf{C}$ thought of as a group with
operation \[ g_1\cdot g_2 = g_1 + g_2 + \frac{1}{2}[g_1,g_2]. \]
Then $G$ is a Lie group with Lie algebra $\mathfrak{g}$, and $G$
contains the subgroup $G_{CM} = H\times\mathbf{C}$ which has Lie
algebra $\mathfrak{g}_{CM}$.  See Section \ref{s.semi-inf} for
definitions and details.

Now let $\{B_t\}_{t\ge0}$ be a Brownian motion on $W$.  The
solution to the stochastic differential equation
\begin{equation}
\label{e.1}
dg_t = g_t\circ dB_t\qquad \text{ with } g_0=e
\end{equation}
is a Brownian motion on $G$, which is given explicitly in Proposition
\ref{p.Mn} and Definition \ref{d.bm}.  For all $t>0$, let
$\nu_t=\mathrm{Law}(g_t)$ denote the heat kernel measure at time
$t$.  If $W$ is finite dimensional, then (\ref{e.h}) implies that
$\mathrm{span}\{(\xi_i,0),[(\xi_i,0),(\xi_j,0)]\}=\mathfrak{g}$,
where $\{\xi_i\}_{i=1}^{\mathrm{dim}(W)}$ is some orthonormal basis
of $W$, and thus we would have satisfaction of H\"ormander's condition
implying that $\nu_t$ is absolutely continuous with respect to Haar
measure on $G=W\times\mathbf{C}$ and its density is a smooth function
on $G$.  If $W$ is infinite-dimensional, then the notion of
subellipticity is not so well defined as there is no canonical
reference measure. But we say that $\nu_t$ is formally subelliptic
(or hypoelliptic) in the sense that all appropriate finite-dimensional
projections (which will be discussed subsequently) are subelliptic.
Similar ``definitions'' of subellipticity in infinite dimensions
have been taken in \cite{BM07,FLT08,MP06}, for example.

Let $\mathrm{Proj}(W)$ denote the collection of finite rank continuous
linear maps $P:W\rightarrow H$ so that $P|_H$ is orthogonal projection.
Further, let $G_P:=PW\times\mathbf{C}$ which is a subgroup of
$G_{CM}$.  For each $P\in\mathrm{Proj}(W)$, $G_P$ is a finite-dimensional Lie group and Brownian motion on $G_P$ is defined
analogously to how it is defined on $G$.  The finite-dimensional
heat kernel measures $\nu_t^P$ will play an important role in the
sequel.  In particular, under the assumption that $[PW,PW]=\mathbf{C}$,
H\"ormander's theorem implies that $d\nu_t^P(x)=p^P_t(x)\,dx$, where
$p^P_t$ is a smooth density and $dx$ is finite-dimensional Haar
measure.

As has been the case in previous infinite-dimensional contexts
\cite{Cecil08,DG08-3,Gordina00-2,Gordina00-1,Gordina02}, our results
actually take the form of two unitary isomorphisms: the ``skeleton''
or ``restriction'' map and the Taylor map on ``square integrable
holomorphic functions'' on $G_{CM}$.

\subsubsection{The restriction isomorphism theorem}
We must first define the Hilbert spaces involved.
Let $\mathcal{H}(G)$ and $\mathcal{H}(G_{CM})$ denote the holomorphic
functions on $G$ and $G_{CM}$ respectively.  Let $\mathcal{P}$ be the space of
holomorphic cylinder polynomials on $G$.  Then Proposition \ref{p.rho2} implies that
$\mathcal{P}\subset L^2(\nu_t)$, and so for $t>0$ define $\mathcal{H}_t^2(G) :=
L^2(\nu_t)$-closure of $\mathcal{P}$.  For $f\in\mathcal{H}(G_{CM})$, let
\[ \|f\|_{\mathcal{H}_t^2(G_{CM})}
	:= \sup_{P\in\mathrm{Proj}(W)} \|f\left|_{G_P}\right.\|_{L^2(\nu_t^P)}
\]
and $\mathcal{H}_t^2(G_{CM}) := \{f\in\mathcal{H}(G_{CM}):
\|f\|_{\mathcal{H}_t^2(G_{CM})}<\infty \}$.  It is proved in Proposition
\ref{p.CM0} that as usual $\nu_t(G_{CM})=0$; however, $\mathcal{H}_t^2(G_{CM})$ should
still be roughly thought of as  $\nu_t$-square integrable holomorphic
functions on $G_{CM}$.  Having made these definitions, we can state our first
theorem.

\begin{thm}
\label{t.1}
For all $t>0$, there is a map $R_t:\mathcal{H}_t^2(G)\rightarrow\mathcal{H}_t^2(G_{CM})$ such
that $R_t$ is an isometric isomorphism, $R_tp=p|_{G_{CM}}$ for any $p\in\mathcal{P}$,
and
\[ |(R_tf)(g)|\le \|f\|_{L^2(\nu_t)} e^{d_h(e,g)^2/2t},
	\text{ for all } g\in G_{CM}, \]
where $d_h$ is the horizontal distance on $G_{CM}$ (see Notation
\ref{n.length}).
\end{thm}

The proof of the pointwise bound and that $R_t$ is actually restriction on
$\mathcal{P}$ are in Theorem \ref{t.Rlin}.  The proof of the isometry
and surjectivity are in Theorem \ref{t.Rsurj}.

\subsubsection{The Taylor isomorphism theorem}
Now let $T(\mathfrak{g}_{CM})$ be the algebraic tensor algebra over
$\mathfrak{g}_{CM}$, $T(\mathfrak{g}_{CM})'$ be its algebraic dual,
$J=J(\mathfrak{g}_{CM})$ be the two-sided ideal in $T(\mathfrak{g}_{CM})$
generated by
\[ \{h\otimes k - k\otimes - [h,k]: h,k\in\mathfrak{g}_{CM} \}, \]
and $J^0=\{\alpha\in T(\mathfrak{g}_{CM})':\alpha(J)=0\}$ be the backwards annihilator of
$J$.  For $t>0$, define
\begin{equation}
\label{e.2}
\|\alpha\|_t^{2}
	:= \sum_{k=0}^\infty \frac{t^k}{k!}\sum_{\xi_1,\ldots,\xi_k\in\Gamma}
		|\langle \alpha, (\xi_1,0)\otimes\cdots\otimes (\xi_k,0)\rangle|^2,
\end{equation}
where $\Gamma$ is an orthonormal basis of $H$, and let $J_t^0 := \{\alpha\in J^0:\|\alpha\|_t<\infty \}$.
Given $f\in\mathcal{H}(G)$, let $\hat{f}\left(e\right)$ denote the
element of $J^0$ defined by $\langle \hat{f}\left(e\right),1\rangle=f\left(e\right)$ and
\[ \langle \hat{f}\left(e\right), h_1\otimes\cdots\otimes h_n\rangle
	= \left(\tilde{h}_1\cdots\tilde{h}_nf\right)(e), \text{ for all }
		h_1,\ldots,h_k\in\mathfrak{g}_{CM} \]
where $\tilde{h}_i$ is the left invariant vector field on $G_{CM}$ such that
$\tilde{h}_i(e)=h_i$.  For $f\in\mathcal{H}_t^2(G_{CM})$, let
$\mathcal{T}_tf=\hat{f}\left(e\right)$.

\begin{thm}
\label{t.2}
For all $t>0$, the map $\mathcal{T}_t:\mathcal{H}_t^2(G_{CM})\rightarrow
J_t^0(\mathfrak{g}_{CM})$ is an isometric isomorphism.
\end{thm}

The isometry in Theorem \ref{t.2} is proved in Proposition \ref{p.Tiso}
and the surjectivity is proved in Theorem \ref{t.surj}.  The
combination of Theorems \ref{t.1} and \ref{t.2} implies that the
mapping $f\mapsto (\mathcal{T}_t\circ R_t)f = \widehat{R_tf}\left(e\right)$, where
\[ \left\langle \widehat{R_tf}\left(e\right), h_1\otimes\cdots\otimes h_k\right\rangle
	= \left(\tilde{h}_1\cdots\tilde{h}_k R_tf\right)(e),
		\text{ for all } h_1,\ldots,h_k\in\mathfrak{g}_{CM},
		\]
is a unitary equivalence between $\mathcal{H}_t^2(G)$ and $J_t^0$.

The organization of the paper is as follows.  In Section \ref{s.1},
the definitions of infinite-dimensional Heisenberg-like Lie algebras
and groups are revisited.  This includes a brief review of complex
abstract Wiener spaces in Section \ref{s.h2}.  In Section \ref{s.deriv}
we explore the relationship between linear and left invariant
derivatives on $G$ which will later be useful in several limiting
arguments.  In Section \ref{s.length}, we prove that the homogeneous
norm and horizontal distance topologies are equivalent.  This fact
is necessary to make use of the finite-dimensional projection groups
introduced in Section \ref{s.gpproj} as approximations to $G$.  In
Section \ref{s.BM}, we define the subelliptic diffusion $\{g_t\}_{t\ge0}$
and its heat kernel measure $\nu_t$ and review various properties
that will be necessary for the sequel.  Most of these properties
follow directly from properties for the nondegenerate heat kernel
measures treated in \cite{DG08-2} and \cite{DG08-3}.  Also, in
Section \ref{s.h5.2} we review the notion of holomorphic functions
in this infinite-dimensional setting.

Section \ref{s.Taylor} gives the proof of the Taylor
isomorphism theorem, including a proof in Section \ref{s.NCF} that
the semi-norm defined in (\ref{e.2}) is in fact a norm.  The proofs in this
section are mostly standard.

In Section \ref{s.restriction}, the restriction map is constructed
and we prove its isometry and surjectivity properties.  Here the
proofs are complicated by several factors, including the use of the
horizontal distance and the fact that the norm defining $J_t^0$ is
not the full Hilbert-Schmidt norm as is used in the nondegenerate
case.  Ultimately, the overall steps here are analogous to
those in the nondegenerate setting, but the proofs are necessarily
adjusted to account for these complications.

\subsection{Discussion of open questions}

Recall that \cite{DG08-3} treated the case of nondegenerate heat
kernel measures on the same infinite-dimensional Heisenberg-like
groups considered here.  One of the main ingredients used there was
the quasi-invariance of the heat kernel measure under shifts by
elements of the Cameron-Martin subgroup. In particular, this allowed
the skeleton or restriction map from $\mathcal{H}_t^2(G)$ to
$\mathcal{H}_t^2(G_{CM})$ to be defined via quasi-invariance.  At
the time of the writing of the present paper, quasi-invariance
results for the subelliptic heat kernel measure were unknown.  Thus,
the construction of the restriction map given here does not rely
on quasi-invariance.  This construction is analogous to that in
\cite{Cecil08}, which treats the case of nondegenerate heat kernel
measures on complex path groups, a case in particular where
quasi-invariance results are not known. After the present paper was accepted,
a quasi-invariance result for the subelliptic heat kernel measure
in this setting was proved in \cite{BGM11}.  Thus, it may now be
possible to give a different proof of our results including the
skeleton map defined via quasi-invariance.

One should also comment that the assumption that
$\mathrm{dim}(\mathbf{C})<\infty$ is necessary at several points.
For example, it is used in an essential way for the proof that the
homogeneous norm topology is equivalent to that of the horizontal
distance. Some readers might be concerned that this restriction on
the dimension of the center means that this subelliptic example is
in some sense only finitely many steps from being elliptic. This
concern would be justified if the Lie bracket is non-trivial on
only a finite-dimensional subspace of $W$, as then the solution to
(\ref{e.1}) is somehow only a finite-dimensional subelliptic diffusion
coupled with an infinite-dimensional flat Brownian motion. However,
if the Lie bracket is in fact non-trivial on an infinite-dimensional
subspace of $W$, then this does introduce several non-trivial
complications, for example, in the proof of equivalence of topologies
and more generally in working with the horizontal distance and
``projections'' of horizontal paths.

Another interesting question is to try to generate holomorphic
functions similar to how it was done in \cite[Section 4]{DGSC09-1}.
Even though one of the techniques of that section, the Fourier-Wigner
transform, has been studied in infinite dimensions (for example,
\cite{Gross1974a}), it is still not clear how  this question can
be approached for infinite-dimensional Heisenberg groups.

{\bf Acknowledgements.} The authors thank Bruce Driver for several
helpful conversations during the writing of this paper.  We would
also like to thank the Mathematics Department at Cornell University,
where much of this research was completed.

\section{Infinite-dimensional complex Heisenberg-like groups}
\label{s.1}

\subsection{Complex abstract Wiener spaces}
\label{s.h2}

Let us first briefly recall the definition of a complex abstract
Wiener space.  We record here only the basic construction and some
standard facts that will be useful for the sequel.  For more
details, see for example Section 2 of \cite{DG08-3} and
its references.

Suppose that $W$ is a complex separable Banach space and
$\mathcal{B}_W$ is the Borel $\sigma$-algebra on $W$. Let
$W_{\operatorname{Re}}$ denote $W$ thought of as a real Banach
space. For $\lambda\in\mathbb{C}$,  let $M_{\lambda}:W\rightarrow W$
be the operation of multiplication by $\lambda$.
\begin{defn}
\label{d.h2.1} A measure $\mu$ on $(W,\mathcal{B}_W)$ is called a
(mean zero, non-degenerate) {\em Gaussian measure} provided that
its characteristic functional is given by
\[
\hat{\mu}(u) := \int_W e^{iu(w)} d\mu(w)
=e^{-\frac{1}{2}q(u,u)}, \text{ for all }u\in W_{\operatorname{Re}}^*,
\]
where $q=q_{\mu}:W_{\operatorname{Re}}^*\times W_{\operatorname{Re}
}^*\rightarrow\mathbb{R}$ is an inner product on $W_{\operatorname{Re}
}^*$. If in addition, $\mu$ is invariant under multiplication
by $i$, that is, $\mu\circ M_{i}^{-1}=\mu$,  we say that $\mu$ is a
{\em complex Gaussian measure }on $W$.
\end{defn}

\begin{thm}
\label{t.h2.3}
Let $\mu$ be a complex Gaussian measure on a complex separable Banach space
$W$.  For $1\le p <\infty$, let
\begin{equation}
\label{e.Cp}
C_p := \int_W \|w\|_W^p \,d\mu(w) < \infty
\end{equation}
For $w\in W$, let
\[ \|w\|_H := \sup _{u\in W^*\setminus\{0\}} \frac{|u(w)|}{\sqrt{q(u,u)}}, \]
and define the {\em Cameron-Martin subspace} $H\subset W$ by
\[ H := \{ h\in W : \|h\|_H<\infty \}. \]
\begin{enumerate}
\item For all $1\le p<\infty$, $C_p<\infty$.

\item $H$ is a dense complex subspace of $W$.

\item There exists a unique inner product, $\langle \cdot,\cdot
\rangle_H$,  on $H$ such that $\|h\|_H^2 = \langle h,h\rangle_H$
for all $h\in H$. Moreover,
with this inner product $H$ is a separable complex Hilbert
space.

\item For any $h\in H$,
\begin{equation}
\label{e.wh}
\| h\|_W \le \sqrt{C_2}\|h\|_H.
\end{equation}
\end{enumerate}
\end{thm}

\begin{notation}
The triple $(W,H,\mu)$ appearing in Theorem \ref{t.h2.3} will be called a {\em
complex abstract Wiener space}.
\end{notation}

We will also need the following facts about linear maps from $W$ into a complex
Hilbert space $K$.  The proof of
the next lemma may be found as part of Lemma 2.6 in \cite{DG08-3}.

\begin{lem}
\label{l.lin}
If $\varphi:W\rightarrow K$ is a linear map, then
\[ \int_W \|\varphi(w)\|_K^2\,d\mu(w)
	= 2\|\varphi\|_{H^*\otimes K}^2. \]
\end{lem}

Now suppose that $\rho:W\times W\rightarrow K$ is a continuous bilinear map
so that
\[ \|\rho\|_0 := \sup\{\rho(w,w')\|_K : \|w\|_W=\|w'\|_W=1\} <\infty. \]
The continuity of $\rho$ and Lemma \ref{l.lin} give the following
proposition which is analogous to Proposition
3.14 in \cite{DG08-2}.

\begin{prop}
\label{p.norm}
The bilinear form $\rho:H\times H\rightarrow K$ is Hilbert-Schmidt;
that is, for any orthonormal basis $\{\xi_j\}_{j=1}^\infty$ of $H$,
\[ \|\rho\|_{HS}^2 := \sum_{j,k=1}^\infty \|\rho(\xi_j,\xi_k)\|_K^2
	<\infty \]
(where $\|\cdot\|_{HS}^2$ is independent of basis).
\end{prop}

\begin{proof}
By Lemma \ref{l.lin},
\begin{align*}
\|\rho(w,\cdot)\|^2_{H^*\otimes K}
	&= \frac{1}{2}\int_W \|\rho(w,w')\|_K^2\,d\mu(w') \\
	&\le \frac{1}{2}\|\rho\|_0^2\|w\|_W^2 \int_W \|w'\|_W^2 \,d\mu(w')
	= \frac{1}{2}C_2\|\rho\|_0^2\|w\|_W^2,
\end{align*}
where $C_2<\infty$ is as defined in (\ref{e.Cp}).
Similarly, viewing $w\mapsto\rho(w,\cdot)$ as a continuous linear map from
$W$ to $H^*\otimes K$,
\begin{align*}
\|\rho\|_{HS}^2
	&= \|h\mapsto\rho(h,\cdot)\|_{H^*\otimes(H^*\otimes K)}^2
	= \frac{1}{2}\int_W \|\rho(w,\cdot)\|_{H^*\otimes K}^2\,d\mu(w) \\
	&\le \frac{1}{4}\int_W C_2\|\rho\|_0^2\|w\|_W^2\,d\mu(w)
	= \frac{1}{4}C_2^2\|\rho\|_0^2.
\end{align*}
\end{proof}

\subsection{Infinite-dimensional complex Heisenberg-like groups}
\label{s.semi-inf}

In this section, we revisit the definition of the infinite-dimensional complex
Heisenberg-like groups constructed in \cite{DG08-3}.  Note that since we are
interested in subelliptic heat kernel measures on these groups, there are some
necessary modifications to the topology.  First we set the following notation
which will hold for the entirety of this paper.

\begin{notation}
Let $(W,H,\mu)$ be a complex abstract Wiener space.  Let $\mathbf{C}$
be a complex Hilbert space with inner product
$\langle\cdot,\cdot\rangle_\mathbf{C}$ and
$\mathrm{dim}(\mathbf{C})=N<\infty$.  Let $\omega:W\times
W\rightarrow\mathbf{C}$  be a continuous skew-symmetric bilinear
form on $W$.  We will also trivially assume that $\omega$
is surjective (otherwise, we just restrict to a linear subspace of
$\mathbf{C}$).

\end{notation}

\begin{defn}
Let $\mathfrak{g}$ denote $W\times\mathbf{C}$ when thought of as a Lie algebra
with the Lie bracket given by
\[ [(X_1,V_1), (X_2,V_2)] := (0, \omega(X_1,X_2)). \]
Let $G$ denote $W\times\mathbf{C}$ when thought of as a group with
multiplication given by
\begin{equation}
\label{e.3.1}
 g_1 g_2 := g_1 + g_2 + \frac{1}{2}[g_1,g_2],
\end{equation}
where $g_1$ and $g_2$ are viewed as elements of $\mathfrak{g}$. For $g_i=(w_i,c_i)$, this may be written equivalently as
\begin{equation}
\label{e.3.2}
(w_1,c_1)\cdot(w_2,c_2) = \left( w_1 + w_2, c_1 + c_2 +
    \frac{1}{2}\omega(w_1,w_2)\right).
\end{equation}
We will call $G$ constructed in this way a {\em Heisenberg-like group}.
\end{defn}
It is easy to verify that, given this bracket and multiplication,
$\mathfrak{g}$ is indeed a Lie algebra and $G$ is a group.
Note that $g^{-1}=-g$ and the identity $e=(0,0)$.

\begin{notation}
Let $\mathfrak{g}_{CM}$ denote $H\times\mathbf{C}$ when thought of as a Lie
subalgebra of $\mathfrak{g}$, and we will refer to $\mathfrak{g}_{CM}$ as the
{\em Cameron-Martin subalgebra} of $\mathfrak{g}$. Similarly, let $G_{CM}$
denote $H\times\mathbf{C}$ when thought of as a subgroup of $G$, and we will
refer to $G_{CM}$ as the {\em Cameron-Martin subgroup} of $G$.
\end{notation}
We will equip $\mathfrak{g}=G$ with the homogeneous norm
\[ \|(w,c)\|_{\mathfrak{g}} := \sqrt{\|w\|_W^2 + \|c\|_\mathbf{C}}, \]
and analogously on $\mathfrak{g}_{CM}=G_{CM}$ we define
\[ \|(A,a)\|_{\mathfrak{g}_{CM}} := \sqrt{\|A\|_H^2 + \|a\|_\mathbf{C}}. \]

\begin{lem}
\label{l.top}
$G$ and $G_{CM}$ are topological groups with respect to the
topologies induced by the homogeneous norms.
\end{lem}

\begin{proof} This is proved similarly to \cite[Lemma 3.3]{DG08-2}.
Since $g^{-1}=-g$, the map $g\mapsto g^{-1}$
is continuous in the $\mathfrak{g}$ and $\mathfrak{g}_{CM}$ topologies.
Also $(g_{1},g_{2})\mapsto\left[g_{1},g_{2}\right]$ and $(g_{1},g_{2})\mapsto
g_{1}+g_{2}$ are continuous in both the $\mathfrak{g}$ and $\mathfrak{g}_{CM}$
topologies.  Thus, it follows from Equation \eqref{e.3.1} that $\left(
g_{1},g_{2}\right)  \mapsto g_{1}\cdot g_{2}$ is continuous as well.
\end{proof}

Before proceeding,
let us give the basic motivating examples for the construction of
these infinite-dimensional Heisenberg-like groups.
\begin{example} [Finite-dimensional complex Heisenberg group]
\label{ex.Heis}
Let $W=H=\mathbb{C}^n\times\mathbb{C}^n$,
$\mathbf{C}=\mathbb{C}$, and
\[ \omega( (w_1,w_2),(z_1,z_2)):=w_1\cdot z_2-w_2\cdot z_1. \]
Then $G=\mathbb{C}^{2n}\times\mathbb{C}$ equipped with a group operation as
defined in (\ref{e.3.2}) is a finite-dimensional complex
Heisenberg group.
\end{example}

\begin{example}
[Heisenberg group of a symplectic vector space]
\label{ex.infHeis}
Let $(K,\langle\cdot,\cdot\rangle)$ be a complex Hilbert space and $Q$ be a
strictly positive trace class operator on $K$.  For $h,k\in K$, let $\langle
h,k\rangle_Q:= \langle h,Qk\rangle$ and $\|h\|_Q:= \sqrt{\langle
h,h\rangle_Q}$, and let $(K_Q,\langle\cdot,\cdot\rangle_Q)$ denote the Hilbert
space completion of $(K,\|\cdot\|_Q)$.  Further assume that $K$ is equipped
with a conjugation $k\mapsto\bar{k}$ which is isometric and commutes with $Q$.
Let $W=K_Q\times K_Q$, $H=K\times K$, and $\omega:W\times
W\rightarrow\mathbb{C}$ be defined by
\[ \omega( (w_1,w_2),(z_1,z_2) ) = \langle w_1,\bar{z}_2\rangle_Q - \langle
	w_2,\bar{z}_1\rangle_Q. \]
Then $G=(K_Q\times K_Q)\times\mathbb{C}$ equipped with a group operation as
defined in (\ref{e.3.2}) is an infinite-dimensional complex
Heisenberg-like group.
\end{example}

\subsection{Derivatives on $G$}
\label{s.deriv}

For $g\in G$, let $L_g:G\rightarrow G$ and $R_g:G\rightarrow G$
denote left and right multiplication by $g$, respectively.  As $G$
is a vector space, to each $g\in G$ we can associate the tangent
space $T_g G$ to $G$ at $g$, which is naturally isomorphic to $G$.

\begin{notation}[Linear and group derivatives]
For $f:G\rightarrow\mathbb{C}$,
$x\in G$, and $h\in\mathfrak{g}$, let
\[ f'(x)h := \partial_h f(x) = \frac{d}{dt}\bigg|_0f(x+th),\]
whenever this derivative exists. More generally, for $h_1,\ldots,h_n\in\mathfrak{g}$, let
\[ f^{(n)}(x)(h_1\otimes\cdots\otimes h_n)
	:= \partial_{h_1}\cdots\partial_{h_n}f(x). \]
For $v,x\in G$, let $v_x \in T_x G$ denote the tangent vector
satisfying $v_xf=f'(x)v$.  If $x(t)$ is any smooth curve in
$G$ such that $x(0) = x$ and $\dot{x}(0)=v$ (for example,
$x(t) = x+tv$), then
\[ L_{g*} v_x = \frac{d}{dt}\bigg|_0 g\cdot x(t). \]
In particular, for $x=e$ and $v_e=h\in\mathfrak{g}$, let
$\tilde{h}(g):=L_{g*}h$, so that $\tilde{h}$ is the unique left invariant
vector field on $G$ such that $\tilde{h}(e)=h$.  We view
$\tilde{h}$ as a first order differential operator acting on smooth
functions by
\[ (\tilde{h}f)(g) = \frac{d}{dt}\bigg|_0 f(g\cdot \sigma(t)), \]
where $\sigma(t)$ is a smooth curve in $G$ such that $\sigma(0)=e$ and
$\dot{\sigma}(0)=h$ (for example, $\sigma(t)=th$).
\end{notation}

The following proposition is Proposition 3.7 of \cite{DG08-2} and
a special case of Proposition 3.16 of \cite{M09}.  The proof is a
simple computation and is included here for the reader's convenience.

\begin{prop}
\label{p.hg}
For $g,x\in G$ and $v_x\in T_x G$,
\[
L_{g*}v_x = v + \frac{1}{2}[g,v],
\]
and this expression does not depend on $x$.
In particular, taking $x=e$, $g=(w,c)$, and $v_e=h=(A,a)\in\mathfrak{g}$ gives
\[ \tilde{h}(g) = \left(A, a + \frac{1}{2}\omega(w,A)\right). \]
\end{prop}

\begin{proof}
Let $x(t)= x+tv$.  Then
\[
L_{g*}v_x = \frac{d}{dt}\bigg|_0 g\cdot x(t)
	= \frac{d}{dt}\bigg|_0 g + x(t) + \frac{1}{2}[g, x(t)]
	= v + \frac{1}{2}[g,v].
\]
\end{proof}

In the sequel, it will be useful to have an expression for the left
invariant derivatives of a smooth function on $G$ in terms of its linear
derivatives.  To do this, we first set the following notation.
\begin{notation}
\label{n.part}
For $k\in\mathbb{N}$, let
\[ \Lambda^k := \{\text{partitions } \theta \text{ of } \{1,\ldots,k\}: \text{
for all } A\in\theta, \#A\le 2\}. \]
If $\{i,j\}\in\theta\in\Lambda^k$, we will always assume without loss of
generality that $i>j$.  For $\ell=0,\ldots,\lfloor k/2\rfloor$, let
\[ \Lambda_\ell^k := \{\theta\in\Lambda^k : \#\{A\in\theta:\#A=2\}=\ell\}.  \]
\end{notation}

\begin{prop}
\label{p.lder}
For $g\in G$, $h\in\mathfrak{g}$, and $f:G\rightarrow\mathbb{C}$ a
smooth function,
\begin{equation}
\label{e.lder1}
\tilde{h}f(g) = f'(g)\tilde{h}(g).
\end{equation}
More generally, for $k\in\mathbb{N}$ and
$h_1,\ldots,h_k\in\mathfrak{g}$,
\begin{equation}
\label{e.lder}
\tilde{h}_k\cdots\tilde{h}_1f(g)
	= \sum_{j=\lceil k/2\rceil}^k f^{(j)}(g)
		\left(\sum_{\theta\in\Lambda_{k-j}^k}
		(h_k,\ldots,h_1)^{\otimes\theta}(g)\right),
\end{equation}
where, for $\theta = \{ \{i_1,i_2\},\ldots,\{i_{2\ell-1},
i_{2\ell}\},\{i_{2\ell+1}\},\ldots,\{i_k\}\}\in\Lambda^k_\ell$,
\[
(h_k,\ldots,h_1)^{\otimes\theta}(g)
	:= \frac{1}{2^\ell}[h_{i_1},h_{i_2}]\otimes \cdots\otimes
		[h_{i_{2\ell-1}},h_{i_{2\ell}}]\otimes
		\tilde{h}_{i_{2\ell+1}}(g) \otimes\cdots \otimes\tilde{h}_{i_k}(g).
\]
\end{prop}
\begin{proof}
The first assertion holds by Proposition \ref{p.hg} and an application of the
chain rule.  Equation (\ref{e.lder}) may be then proved by induction.  So
assume the formula holds for $k$ and consider $k+1$.
\begin{align*}
\tilde{h}_{k+1}\tilde{h}_k\cdots&\tilde{h}_1f(g)
	= \frac{d}{dt}\bigg|_0 \tilde{h}_k\cdots\tilde{h}_1f(g\cdot th_{k+1}) \\
	&= \frac{d}{dt}\bigg|_0
		\sum_{j=\lceil k/2\rceil}^k f^{(j)}(g\cdot th_{k+1})
		\sum_{\theta\in\Lambda_{k-j}^k} (h_k,\ldots,h_1)^{\otimes\theta}
		(g\cdot th_{k+1}) \\
	&= \sum_{j=\lceil k/2\rceil}^k f^{(j+1)}(g)
		\sum_{\theta\in\Lambda_{k-j}^k} \tilde{h}_{k+1}(g)\otimes
		(h_k,\ldots,h_1)^{\otimes\theta}(g) \\
	&\qquad + \sum_{j=\lceil k/2\rceil}^k f^{(j)}(g)
		\sum_{\theta\in\Lambda_{k-j}^k} \frac{d}{dt}\bigg|_0
		(h_k,\ldots,h_1)^{\otimes\theta}(g\cdot th_{k+1}).
\end{align*}
For $g=(w,c)$, $h=(A,a)$, and $k=(B,b)$,
\[ \frac{d}{dt}\bigg|_0 \tilde{h}(g\cdot tk)
	= \left( A, a + \frac{1}{2}\omega(w+tB,A)\right)
	= \left(0,\frac{1}{2}\omega(B,A)\right)
	= \frac{1}{2}[k,h], \]
which is independent of $g$.  (Note that $\widetilde{[k,h]}(g)=[k,h]$.)
Thus, for $\theta = \{ \{i_1,i_2\},\ldots,\{i_{2\ell-1},
i_{2\ell}\},\{i_{2\ell+1}\},\ldots,\{i_k\}\}\in\Lambda^k_\ell$,
\begin{align*}
&\frac{d}{dt}\bigg|_0(h_k,\ldots,h_1)^{\otimes\theta}(g\cdot th_{k+1}) \\
	&= \frac{d}{dt}\bigg|_0 \frac{1}{2^\ell}\bigg\{
		[h_{i_1},h_{i_2}]\otimes \cdots\otimes
		[h_{i_{2\ell-1}},h_{i_{2\ell}}] \otimes\tilde{h}_{i_{2\ell+1}}(g\cdot th_{k+1}) \otimes\cdots
		\otimes\tilde{h}_{i_k}(g\cdot th_{k+1}) \bigg\}\\
	&= \sum_{j=2\ell+1}^k \frac{1}{2^{\ell+1}}
		\bigg([h_{i_1},h_{i_2}]
		\otimes \cdots\otimes[h_{i_{2\ell-1}},h_{i_{2\ell}}] \\
	&\qquad \qquad \otimes \tilde{h}_{i_{2k+1}}(g) \otimes\cdots\otimes
		\tilde{h}_{j-1}(g)\otimes [h_{k+1},h_j]\otimes
		\tilde{h}_{j+1}(g)\otimes \cdots\otimes \tilde{h}_{i_k}(g) \bigg).
\end{align*}
Rearranging terms and indices gives the desired formula.
\end{proof}

Let us write out (\ref{e.lder}) for the first few $n$.  The expression for
$n=1$ is already given in equation (\ref{e.lder1}).  For $n=2$ and $n=3$,
we have
\begin{align}
\label{e.lder2}
\tilde{h}_2\tilde{h}_1f(g)
	&= f''(g)\left(\tilde{h}_2(g)\otimes\tilde{h}_1(g)\right)
		+ \frac{1}{2}f'(g)[h_2,h_1] \\
\label{e.lder3}
\tilde{h}_3\tilde{h}_2\tilde{h}_1f(g)
	&= f'''(g)\left(\tilde{h}_3(g) \otimes \tilde{h}_2(g)
			\otimes \tilde{h}_1(g)\right) \\
	&\notag \hspace*{-2ex}
		+ \frac{1}{2}f''(g)\left([h_3,h_2]\otimes \tilde{h}_1(g)
		+ [h_3,h_1]\otimes \tilde{h}_2(g)
		+ [h_2,h_1]\otimes \tilde{h}_3(g)\right).
\end{align}
In particular, (\ref{e.lder2}) implies that, for $h,k\in\mathfrak{g}$,
\begin{equation}
\label{e.hk}
\left(\tilde{h}\tilde{k}-\tilde{k}\tilde{h}\right)f = \widetilde{[h,k]}f.
\end{equation}

\subsection{Distances on $G_{CM}$}
\label{s.length}

We define here the sub-Riemannian distance on $G_{CM}$ and show
that the topology induced by this metric is equivalent to the
topology induced by the homogeneous norm $\|\cdot\|_{\mathfrak{g}_{CM}}$.
Note that in finite dimensions this result is standard and is usually
proved via compactness arguments (see for example Chapter 5 of
\cite{BLU07}).  Of course, these arguments are invalid in infinite
dimensions and so we resort to more direct methods of proof.  Note
that the results in this subsection rely directly on the fact that
$N=\mathrm{dim}(\mathbf{C})<\infty$.

\begin{notation}
(Riemannian and horizontal distances)
\label{n.length}

\begin{enumerate}
\item Let $C^1_{CM}$ denote the set of $C^1$-paths $\sigma:[0,1]\rightarrow
G_{CM}$.

\item For $x=(A,a)\in G_{CM}$, let
\[ |x|_{\mathfrak{g}_{CM}}^2 : = \|A\|_H^2 + \|a\|_\mathbf{C}^2. \]
The {\em length} of a $C^1$-path $\sigma:[a,b]\rightarrow
G_{CM}$ is defined as
\[ \ell(\sigma)
	:= \int_a^b |L_{\sigma^{-1}(s)*}\dot{\sigma}(s)|_{\mathfrak{g}_{CM}} \,ds.
\]

\item The {\em Riemannian distance} between $x,y\in G_{CM}$ is defined by
\[ d_{CM}(x,y) := \inf\{\ell(\sigma): \sigma\in C^1_{CM} \text{ such
    that } \sigma(0)=x \text{ and } \sigma(1)=y \}. \]

\item \label{i.2}
A $C^1$-path $\sigma:[a,b]\rightarrow G_{CM}$ is {\em horizontal} if
$L_{\sigma(t)^{-1}*}\dot{\sigma}(t)\in H\times\{0\}$
for a.e.~$t$.  Let $C^{1,h}_{CM}$ denote the set of horizontal paths
$\sigma:[0,1]\rightarrow G_{CM}$.

\item The {\em horizontal distance} between $x,y\in G_{CM}$ is defined by
\[ d_h(x,y) := \inf\{\ell(\sigma): \sigma\in C^{1,h}_{CM} \text{ such
    that } \sigma(0)=x \text{ and } \sigma(1)=y \}. \]
\end{enumerate}
\end{notation}

\begin{remark}
\label{r.horiz}
Note that if $\sigma(t)=(A(t),a(t))\in C_{CM}^{1,h}$, then
\[ L_{\sigma(t)^{-1}*}\dot{\sigma}(t)
	= \left(\dot{A}(t), \dot{a}(t) - \frac{1}{2}\omega(A(t),\dot{A}(t))\right)
		\in H\times\{0\}
\]
implies that $\sigma$ must satisfy
\[ a(t) = a(0) + \frac{1}{2}\int_0^t \omega(A(s),\dot{A}(s))\,ds, \]
and the length of $\sigma$ is given by
\begin{align*}
\ell(\sigma)
	= \int_0^1 |L_{\sigma^{-1}(s)*}\dot{\sigma}(s)|_{\mathfrak{g}_{CM}}\,ds
	= \int_0^1 \|\dot{A}(s)\|_H\,ds .
\end{align*}
\end{remark}

Proposition 3.10 of \cite{DG08-2} gives the following comparison of the
$|\cdot|_{\mathfrak{g}_{CM}}$ and Riemannian metrics.

\begin{prop}
\label{p.3.10}
There exists $\delta=\delta(\omega)>0$ such that, for all $x,y\in G_{CM}$,
\[ d_{CM}(x,y)
	\le \left(1 + \frac{1}{4\delta}|x|_{\mathfrak{g}_{CM}} \wedge
		|y|_{\mathfrak{g}_{CM}}\right)|y-x|_{\mathfrak{g}_{CM}}, \]
and, in particular, $d_{CM}(e,x)\le |x|_{\mathfrak{g}_{CM}}$ for any $x\in G_{CM}$.
Also, there exists $k=k(\omega)<\infty$ such that,
if $x,y\in G_{CM}$ satisfy $d_{CM}(x,y)\le \delta$, then
\[ |y-x|_{\mathfrak{g}_{CM}}
	\le  k(1 + |x|_{\mathfrak{g}_{CM}} \wedge |y|_{\mathfrak{g}_{CM}})
		d_{CM}(x,y). \]
\end{prop}

Proposition \ref{p.3.10} implies for example that the topology
induced by $|\cdot|_{\mathfrak{g}_{CM}}$ is equivalent
to that induced by the Riemannian distance.  For the subelliptic
case, these are of course not the relevant topologies.  However, this
result may be used to prove that the homogenous norm on $\mathfrak{g}_{CM}$
and the horizontal distance are comparable at the identity.  The
following proposition is Theorem C.2 of \cite{DG08-2}.  We record
the proof here for the reader's convenience and to emphasize the
dependence of the upper bound constant $K_2$ on $N={\rm dim}(\mathbf{C})$.

\begin{prop}
\label{p.length}
If $\{\omega(A,B): A,B\in H\}=\mathbf{C}$,
then there exist finite constants $K_1=K_1(\omega)$ and $K_2=K_2(N,\omega)$
such that, for all $(A,a)\in\mathfrak{g}_{CM}$,
\[ K_1\|(A,a)\|_{\mathfrak{g}_{CM}} \le d_h(e,(A,a))
	\le K_2\|(A,a)\|_{\mathfrak{g}_{CM}}. \]
\end{prop}
\begin{proof}
For any left-invariant metric $d$ on $G_{CM}$ (for example $d_{CM}$ or $d_h$),
we have
\begin{equation}
\label{e.12.2}
d(e,xy) \le d(e,x) + d(x,xy) = d(e,x) + d(e,y),
\end{equation}
for all $x,y\in G_{CM}$.  Given any horizontal path $\sigma=(w,c)$ joining
$e$ to $(A,a)$, we have from Remark \ref{r.horiz} that
\[
\ell(\sigma)
	= \int_0^1 \|\dot{w}(s)\|_H \, ds\ge\| A\| _H.
\]
Taking the infimum over all horizontal paths connecting $e$ to $(A,a)$,
it then follows that
\[ d_h(e,(A,a)) \ge \|A\|_H.
\]
Since the path $\sigma(t) = (tA,0)$ is horizontal and
\[ \|A\|_H = \ell(\sigma) \ge d_h(e,(A,0))  \ge \|A\|_H, \]
it follows that
\begin{equation}
\label{e.12.4}
d_h(e,(A,0))
	= \|A\|_H \text{ for all } A\in H.
\end{equation}

Given $A,B\in H$, let $\gamma(t) = A\cos 2\pi t + B\sin 2\pi t$ for $0\le t\le
1$, and consider the path
\[ \sigma(t) = \left(\gamma(t) - A, \frac{1}{2}\int_0^t \omega(\gamma(s) - A,
	\dot{\gamma}(s)) \,ds\right). \]
Note that $\sigma$ is a horizontal curve with
$L_{\sigma(t)^{-1}_*}\dot{\sigma}(t) =
(\dot{\gamma}(t),0)$, $\sigma(0) = e$, and
\[ \sigma(1) = \left(0,\frac{1}{2} \int_0^1
		\omega(\gamma(s),\dot{\gamma}(s)) \, ds \right)
	= \left(0,\pi \int_0^1 \omega(A,B)\, ds \right)
	= (0,\pi\omega(A,B)).
\]
Thus, we may conclude that
\begin{align}
d_h(e,(0,\pi\omega(A,B)))
	\notag
	\le \ell(\sigma)
	&= 2\pi \int_0^1 \|-A\sin 2\pi s + B\cos 2 \pi s\|_H \, ds \\
	&\label{e.12.5}
	\le 2 \pi(\|A\|_H + \|B\|_H). \end{align}

Now choose $\{A_\ell,B_\ell\}_{\ell=1}^{N}\subset H$ such that
$\{\pi\omega(A_\ell,B_\ell)\}_{\ell=1}^{N}$ is a basis for $\mathbf{C}$.
Let $\{\varepsilon^\ell\} _{\ell=1}^{N}$ be the corresponding dual
basis.  Hence, for any $a\in\mathbf{C}$, we have
\begin{align*}
d_h(e,(0,a))
	&= d_h\left(e,\prod_{\ell=1}^N
		(0,\varepsilon^\ell(a)\pi\omega(A_\ell,B_\ell))\right) \\
	&\le \sum_{\ell=1}^N d_h(e,(0,\varepsilon^\ell(a)
		\pi \omega(A_\ell,B_\ell))) \\
	&= \sum_{\ell=1}^N d_h\left(e,\left(0,\pi\omega\left(\mathrm{sgn}
		(\varepsilon^\ell(a)) \sqrt{|\varepsilon^\ell(a)| }A_\ell,
		\sqrt{|\varepsilon^\ell(a)| }B_\ell\right)\right)\right) \\
	&\le 2\pi \sum_{\ell=1}^N \left(
		\left\|\sqrt{|\varepsilon^\ell(a)|}A_\ell\right\|_H
		+ \left\| \sqrt{|\varepsilon^\ell(a)|}B_\ell \right\|_H \right),
\end{align*}
wherein we have used (\ref{e.12.2}) for the first inequality
and (\ref{e.12.5}) for the second inequality.  Then H\"older's inequality
implies that
\begin{equation}
\label{e.12.6}
d_h(e,(0,a))
	\le 4\pi\sum_{\ell=1}^N \sqrt{|\varepsilon^\ell(a)|}
	\le 4\pi C \sqrt{\|a\|_\mathbf{C}},
\end{equation}
for a finite constant $C=C(N,\omega)$.
Combining equations (\ref{e.12.2}), (\ref{e.12.4}), and (\ref{e.12.6})
gives,
\begin{align*}
d_h(e,(A,a)) &= d_h(e,(A,0)(0,a)) \\
	&\le d_h(e,(A,0)) + d_h(e,(0,a)  ) \\
	&\le \|A\|_H + C(N,\omega) \sqrt{\|a\|_\mathbf{C}}
	\le \sqrt{2}\left(1\wedge
		C(N,\omega)\right)\|(A,a)\|_{\mathfrak{g}_{CM}},
\end{align*}
which completes the proof of the upper bound.

To prove the lower bound, consider first the dilations defined by
\[ \varphi_\lambda(w,c) := (\lambda w,\lambda^2c), \quad \text{ for } \lambda>0
\text{ and } (w,c)\in \mathfrak{g}_{CM} = G_{CM}. \]
One easily verifies that $\varphi_\lambda$ is both a Lie
algebra homomorphism on $\mathfrak{g}_{CM}$ and a group homomorphism
on $G_{CM}$.  Using the homomorphism property, it follows that, for any
$C^1$-path $\sigma$,
\[
L_{\varphi_\lambda(\sigma(t))^{-1}*}\frac{d}{dt} \varphi_\lambda(\sigma(t))
	= \varphi_\lambda(L_{\sigma(t)^{-1}*}\dot{\sigma}(t)).
\]
Consequently, if $\sigma$ is a horizontal curve, then $\varphi_\lambda\circ
\sigma$ is again horizontal and $\ell(\varphi_\lambda\circ \sigma) =
\lambda\ell(\sigma)$.  Thus, we may conclude that
\begin{equation}
\label{e.12.8}
d_h(\varphi_\lambda(x),\varphi_\lambda(y))
	= \lambda d_h(x,y),
\end{equation}
for all $x,y \in G_{CM}$.

Now, by the first part of Proposition \ref{p.3.10},
$d_{CM}(e,x)\le|x|_{\mathfrak{g}_{CM}}$, for all $x\in G_{CM}$.
Combining this with the second part of the same proposition implies
that there exist $\delta>0$ and $k<\infty$ such that, if
$|x|_{\mathfrak{g}_{CM}}\le \delta$, then $|x|_{\mathfrak{g}_{CM}}
\le  k d_{CM}(x,y)$.  So, for arbitrary $x=(A,a) \in G_{CM}$, choose
$\lambda=\lambda(x)>0$ so that
\[
\delta^2 = |\varphi_\lambda(x)|_{\mathfrak{g}_{CM}}^2
	= \lambda^2 \|A\|_H^2 + \lambda^4 \|a\|_\mathbf{C}^2;
\]
that is, take
\[
\lambda^2 = \frac{\sqrt{\|A\|_H^4 + 4\|a\| _\mathbf{C}^2\delta^2}
	- \|A\|_H^2}{2\|a\|_\mathbf{C}^2}.
\]
Equation (\ref{e.12.8}) and Proposition \ref{p.3.10} then imply that
\[ \lambda k d_h(e,x)
	= k d_h(e,\varphi_\lambda(x))
	\ge k d_{CM}(e,\varphi_\lambda(x))
	\ge |\varphi_\lambda(x)|_{\mathfrak{g}_{CM}} =  \delta\]
Thus,
\begin{align}
d_h(e,x)^2
	&\nonumber
	\ge \frac{\delta^2}{k^2\lambda^2}
	= \frac{\delta^2}{k^2} \frac{2\|a\|_\mathbf{C}^2}
		{\sqrt{\|A\|_H^4 + 4\delta^2\|a\|_\mathbf{C}^2} - \| A\|_H^2}\\
	&\label{e.12.10}
	= \frac{2\delta^2 \|a\|_\mathbf{C}^2}{k^2\|A\|_H^2}
		\frac{1}{\sqrt{1+\frac{4\delta^2\|a\|_\mathbf{C}^2}
		{\|A\|_H^4}}-1}.
\end{align}
Since $\sqrt{1+x}-1\le\min(x/2,\sqrt{x})$, we have
\[
\frac{1}{\sqrt{1+x}-1}
	\ge \max\left(\frac{2}{x},\frac{1}{\sqrt{x}}\right)
	\ge \frac{1}{x}+\frac{1}{2\sqrt{x}}.
\]
Using this estimate with $x = 4\delta^2\|a\|_\mathbf{C}^2\|A\|_H^{-4}$
in equation (\ref{e.12.10}) shows that
\[
d_h(e,x)^2
	\ge \frac{2\delta^2\|a\|_\mathbf{C}^2}{k^2\|A\| _H^2}
		\left(\frac{\|A\| _H^4}{4\delta^2\|a\|_\mathbf{C}^2} +
		\frac{\|A\|_H^2}{4\delta\| a\| _\mathbf{C}} \right)
	= \frac{1}{2k^2}(\|A\| _H^2+\delta\|a\|_\mathbf{C}),
\]
which implies the lower bound.
\end{proof}

Since $G_{CM}$ is stratified, it turns out that comparability of the metrics
at $e$ is sufficient to imply the equivalence of their respective topologies.

\begin{prop}
\label{p.equivtop}
The topologies generated by $d_h$ and $\|\cdot\|_{\mathfrak{g}_{CM}}$ are
equivalent.
\end{prop}

\begin{proof}
Fix $x=(A,a)\in G_{CM}$.
First note that, by
Proposition \ref{p.length} and the left invariance of the horizontal distance,
there exists $K_1=K_1(\omega)<\infty$ such that, for any $y=(B,b)\in G_{CM}$,
\[ \sqrt{\|B-A\|_H^2+\left\|b-a-\frac{1}{2}\omega(A,B)\right\|_\mathbf{C}}
	= \|x^{-1}y\|_{\mathfrak{g}_{CM}} \le K_1 d_h(e,x^{-1}y) = K_1d_h(x,y). \]
So if $d_h(x,y)<\delta$ for some $\delta>0$, then
\[ \|B-A\|_H \le K_1d_h(x,y) < K_1\delta, \]
and
\begin{align*}
\|b-a\|_\mathbf{C}
	&\le \left\|b-a-\frac{1}{2}\omega(A,B)\right\|_\mathbf{C} +
		\frac{1}{2}\|\omega(A,B)\|_\mathbf{C} \\
	&\le K_1^2d_h(x,y)^2 + \frac{1}{2}\|\omega(A,B-A)\|_\mathbf{C} \\
	&< K_1^2\delta^2 + \frac{1}{2}\|\omega\|_{op}\|A\|_H\|B-A\|_H
	< K_1^2\delta^2 + \frac{1}{2}\|\omega\|_{op}\|A\|_H\delta,
\end{align*}
where
\[ \|\omega\|_{op} := \sup\{\|\omega(h,k)\|_\mathbf{C}:\|h\|_H=\|k\|_H=1\}
	<\infty, \]
by the continuity of $\omega$ and (\ref{e.wh}).  Thus, given any $R\in(0,1)$,
one may clearly choose $c=c(x,\omega)$ sufficiently large (for example,
$c=2(\sqrt{2}K_1+\frac{1}{2}\|\omega\|_{op}\|A\|)$) so that $d_h(x,y)<\delta = R^2/c$
implies that
\begin{align*}
\|y-x\|_{\mathfrak{g}_{CM}}
	&= \sqrt{\|B-A\|_H^2+\|b-a\|_\mathbf{C}} \\
	&< \sqrt{K_1^2\delta^2 + K_1^2\delta^2+\frac{1}{2}\|\omega\|_{op}\|A\|_H\delta} \\
	&= \sqrt{2K_1^2\frac{R^4}{c^2} +
		\frac{1}{2}\|\omega\|_{op}\|A\|_H\frac{R^2}{c}}
	<\sqrt{R^2}
	=R.
\end{align*}

Similarly, the left invariance of $d_h$ and Proposition \ref{p.length} imply
that there exists $K_2=K_2(N,\omega)<\infty$ such that
\[ d_h(x,y) \le K_2\|x^{-1}y\|_{\mathfrak{g}_{CM}}= K_2\sqrt{\|B-A\|_H^2+\left\|b-a
		-\frac{1}{2}\omega(A,B)\right\|_\mathbf{C}}. \]
So if we suppose that
$\|y-x\|_{\mathfrak{g}_{CM}}=\sqrt{\|B-A\|_H^2+\|b-a\|_\mathbf{C}}<\delta'$,
then
\begin{align*}
d_h(x,y) &\le K_2\sqrt{\|B-A\|_H^2+\|b-a\|_\mathbf{C}
		+\frac{1}{2}\|\omega(A,B-A)\|_\mathbf{C}} \\
	&\le K_2\left(\|y-x\|_{\mathfrak{g}_{CM}} +
		\sqrt{\frac{1}{2}\|\omega\|_{op}\|A\|_H\|B-A\|_H}\right) \\
	&\le K_2\left(\delta' + \sqrt{\frac{1}{2}\|\omega\|_{op}\|A\|_H\delta'}\right).
\end{align*}
Again, given any $R\in(0,1)$, one may find $c'=c'(x,N,\omega)$ such
that $\|y-x\|_{\mathfrak{g}_{CM}}<\delta'=R^2/c'$ implies that $d_h(x,y)<R$.
\end{proof}

\subsection{Finite-dimensional projection groups}
\label{s.gpproj}
The finite-dimensional projections of $G$ defined in this section
will be important in the sequel.  Note that the construction of these
projections is quite natural in the sense that they come from the usual projections of the
abstract Wiener space; however, the projections defined here are not group
homomorphisms, which is a complicating factor in some of the following proofs.

As usual, let $(W,H,\mu)$ denote a complex abstract Wiener space.
Let $i:H\rightarrow W$ be the inclusion map, and $i^*:W^*\rightarrow H^*$ be
its transpose so that $i^*\ell:=\ell\circ i$ for all $\ell\in W^*$.  Also,
let
\[ H_* := \{h\in H: \langle\cdot,h\rangle_H\in \mathrm{Range}(i^*)\subset H^*\}.
\]
That is, for $h\in H$, $h\in H_*$ if and only if $\langle\cdot,h\rangle_H\in
H^*$ extends to a continuous linear functional on $W$, which we will continue
to denote by $\langle\cdot,h\rangle_H$.  Because $H$ is a dense subspace of
$W$, $i^*$ is injective and thus has a dense range.  Since
$H\ni h\mapsto\langle\cdot,h\rangle_H\in H^*$ is a
linear isometric isomorphism, it follows that $H_*\ni
h\mapsto\langle\cdot,h\rangle_H\in W^*$ is a linear isomorphism
also, and so $H_*$ is a dense subspace of $H$.

Suppose that $P:H\rightarrow H$ is a finite rank orthogonal projection
such that $PH\subset H_*$.  Let $\{\xi_j\}_{j=1}^m$ be an orthonormal basis for
$PH$.  Then we may extend $P$ to a (unique) continuous operator
from $W\rightarrow H$ (still denoted by $P$) by letting
\begin{equation}
\label{e.proj}
Pw := \sum_{j=1}^m \langle w,\xi_j\rangle_H \xi_j
\end{equation}
for all $w\in W$.

\begin{notation}
\label{n.proj}
Let $\mathrm{Proj}(W)$ denote the collection of finite rank projections
on $W$ such that
\begin{enumerate}
\item $PW\subset H_*$,
\item $P|_H:H\rightarrow H$ is an orthogonal projection (that is, $P$ has the form given in equation
\eqref{e.proj}), and
\item $PW$ is sufficiently large to satisfy
H\"ormander's condition (that is, $\{\omega(A,B):A,B\in PW\}=\mathbf{C}$).
\end{enumerate}
\end{notation}

For each $P\in\mathrm{Proj}(W)$, we may define $G_P:=
PW\times\mathbf{C}\subset H_*\times\mathbf{C}$ and a corresponding
projection $\pi_P:G\rightarrow G_P$ \[ \pi_P(w,x):= (Pw,x). \] We
will also let $\mathfrak{g}_P=\mathrm{Lie}(G_P) = PW\times\mathbf{C}$.

For any $\{P_n\}_{n=1}^\infty\subset\mathrm{Proj}(W)$ such that
$P_n|_H\uparrow I_H$, we may choose a sequence of complex orthonormal bases
$\Gamma_n$ for each
$P_nH$ so that $\Gamma_n\uparrow\Gamma$ a complex orthonormal basis for
$H$.  Thus, for the sequel, we will often consider a sequence of
projections with respect to a fixed orthonormal basis.

\begin{notation}
\label{n.pn}
Let $\{\xi_j\}_{j=1}^\infty\subset H_*$ be a fixed orthonormal basis of
$H$.  We will let $P_n$ denote the corresponding projections onto
$P_nW$, that is,
\[ P_nw = \sum_{j=1}^n \langle w,\xi_j\rangle_H \xi_j. \]
Let $G_n=G_{P_n}$,
$\mathfrak{g}_n=\mathrm{Lie}(G_n)$, and $\pi_n=\pi_{P_n}:G\rightarrow G_n$.
So $\{\pi_n\}_{n=1}^\infty$
is an increasing sequence of projections so that
$\pi_n|_{G_{CM}}\uparrow I|_{G_{CM}}$.
In the sequel, it will also be convenient to let
$\Gamma=\{\eta_j\}_{j=1}^\infty = \{(\xi_j,0)\}_{j=1}^\infty$ denote
a basis of $H\times\{0\}$.

(It is clear that, in order for $P_n\in\mathrm{Proj}(W)$, it will
be necessary to have a minimal $n$ so that
$\mathrm{span}\{\omega(\xi_i,\xi_j):i,j=1,\ldots,n\}=\mathbf{C}$.
However, since these projections will be primarily used for large
$n$ as approximations to $G$, we will ignore this issue in the
sequel and always assume we have a large enough projection.)
\end{notation}

\subsection{Brownian motion on $G$}
\label{s.BM}
Here we define a ``subelliptic'' Brownian motion $\{g_t\}_{t\ge0}$
on $G$ and collect various of its properties that are necessary for
the sequel.  The primary references for this section are \cite{DG08-2,DG08-3}.

Let $\{B_t\}_{t\ge0}$ be a Brownian motion on $W$ with variance determined by
\[
\mathbb{E}\left[\langle B_s,h\rangle_H \langle B_t,k\rangle_H\right]
    = \langle h,k \rangle_H \min(s,t),
\]
for all $s,t\ge0$ and $h,k\in H_*$.  The following is Proposition 4.1 of
\cite{DG08-2} and this result implicitly relies on the fact that Proposition
\ref{p.norm} implies that the bilinear form $\omega$ is a Hilbert-Schmidt.
\begin{prop}
\label{p.Mn}
For $P\in\mathrm{Proj}(W)$, let $M_t^P$ be the continuous $L^2$-martingale on
$\mathbf{C}$ defined by
\[ M_t^P = \int_0^t\omega(PB_s,dPB_s). \]
In particular, if $\{P_n\}_{n=1}^\infty\subset\mathrm{Proj}(W)$
is an increasing sequence of projections as in
Notation \ref{n.pn} and $M_t^n:=M_t^{P_n}$, then there exists an $L^2$-martingale
$\{M_t\}_{t\ge0}$ in $\mathbf{C}$ such that, for all $p\in[1,\infty)$
and $t>0$,
\[ \lim_{n\rightarrow\infty} \mathbb{E}\left[\sup_{\tau\le t}
    \|M_\tau^n-M_\tau\|_\mathbf{C}^p\right]=0, \]
and $M_t$ is independent of the sequence of projections.
\end{prop}

As $M_t$ is independent of the defining sequence of projections, we will
denote the limiting process by
\[
M_t = \int_0^t \omega(B_s,dB_s).
\]

\begin{defn}
\label{d.bm}
The continuous $G$-valued process given by
\[g_t = \left( B_t, \frac{1}{2}M_t\right)
	= \left( B_t, \frac{1}{2}\int_0^t \omega(B_s,dB_s)\right).
\]
is a {\em Brownian motion} on $G$.
For $t>0$, let $\nu_t=\mathrm{Law}(g_t)$ denote the {\em heat kernel measure
at time $t$} on $G$.
\end{defn}

\begin{defn}
\label{d.cyl}
A function $f:G\rightarrow\mathbb{C}$ is a {\em cylinder function}
if it may be written as $f=F\circ\pi_P$, for some $P\in\mathrm{Proj}(W)$
and $F:G_P\rightarrow\mathbb{C}$.  We say that $f$ is a {\em smooth
(holomorphic) cylinder function} if $F$ is smooth (holomorphic).
\end{defn}

\begin{prop}
\label{p.L}
If $f:G\rightarrow\mathbb{C}$ is a smooth cylinder function, let
\[ Lf := \sum_{j=1}^\infty \left[\tilde{\eta}_j^2 +
    \widetilde{i\eta}_j^2\right] f, \]
where $\{\eta_j\}_{j=1}^\infty$ is a basis for
$H\times\{0\}$ as in Notation \ref{n.pn}.  Then $Lf$ is well defined, that is, the
above sum is convergent and independent of basis.  Moreover,
$\frac{1}{4}L$ is the generator for $\{g_t\}_{t\ge0}$, so that
\[ f(g_t) - \frac{1}{4}\int_0^t Lf(g_s)\,ds \]
is a local martingale for any smooth cylinder function $f$.
\end{prop}

Proposition \ref{p.Mn} along with the fact that, for all $p\in[1,\infty)$ and
$t>0$,
\[ \lim_{n\rightarrow\infty} \mathbb{E}\left[\sup_{\tau\le
	t}\|B_\tau-P_nB_\tau\|_W^p \right] = 0 \]
(see for example Proposition 4.6 of \cite{DG08-2})
makes the following proposition clear.

\begin{prop}
\label{p.bmapprox}
For $P\in\mathrm{Proj}(W)$, let $g_t^P$ be the continuous process on $G_P$
defined by
\[ g_t^P = \left(PB_t, \frac{1}{2}\int_0^t\omega(PB_s,dPB_s)\right). \]
Then $g_t^P$ is a Brownian motion on $G_P$ .  In particular, let
$\{P_n\}_{n=1}^\infty\subset\mathrm{Proj}(W)$ be increasing
projections as in
Notation \ref{n.pn} and $g_t^n:=g_t^{P_n}$.  Then, for all $p\in[1,\infty)$
and $t>0$,
\[ \lim_{n\rightarrow\infty} \mathbb{E}\left[\sup_{\tau\le t}
    \|g_\tau^n-g_\tau\|_\mathfrak{g}^p\right]=0. \]
\end{prop}

\begin{notation}
For all $P\in\operatorname*{Proj}\left(  W\right)$ and $t>0$, let $\nu_t^P :=
\mathrm{Law}(g_t^P)$, and for all $n\in\mathbb{N}$ let $\nu_t^n := \mathrm{Law}(g_t^n)
=\mathrm{Law}(g_t^{P_n})$.
\end{notation}

For all projections satisfying H\"ormander's condition,
the Brownian motions on $G_P$ are
true subelliptic diffusions in the sense that their laws are absolutely
continuous with respect to the finite-dimensional reference measure and their
transition kernels are smooth.

\begin{lem}
\label{l.4.8} For all $P\in\operatorname*{Proj}(W)$ and $t>0$, we have
$\nu_{t}^{P}(dx) = p_t^P(e,x)dx$,
where $dx$ is the Riemannian volume measure (equal to Haar measure) and
$p_{t}^{P}(x,y)$ is the heat kernel on $G_P.$
\end{lem}

\begin{proof}
An application of Proposition \ref{p.L} with $G$ replaced by $G_{P}$ implies
that $\nu_t^P=\operatorname*{Law}(g^P_t)$
is a weak solution to the heat equation on $G_P$ with generator
\[ L^Pf := \sum_{j=1}^m \left[\widetilde{(\xi_j,0)}^2 +
    \widetilde{(i\xi_j,0)}^2\right] f\]
for smooth  functions $f:G_P\rightarrow\mathbb{C}$,
where $\{\xi_j\}_{j=1}^m$ is a complex orthonormal basis of $PH$.
The result now follows from the fact that $[PW,PW]=\mathbf{C}$, as this
implies $\left\{(\xi_j,0),(i\xi_j,0)\right\}_{j=1}^m$
satisfies H\"ormander's condition,
and thus $L^P$ is a hypoelliptic operator \cite{Hormander67}.
\end{proof}

The next proposition is a version of Fernique's theorem for the subelliptic heat
kernel measures and follows directly from the proof in the elliptic case (see Theorem
4.16 of \cite{DG08-2}).  In particular, this kind of exponential integrability
result is required to have a nontrivial class of holomorphic square integrable functions.

\begin{prop}[Subelliptic Fernique's theorem]
\label{p.rho}
There exists $\delta>0$ such that, for all $\varepsilon\in(0,\delta)$ and
$t>0$,
\[ \sup_{P\in\mathrm{Proj}(W)} \int_{G_P}
	    e^{\varepsilon\|g\|_\mathfrak{g}^2/t}\,d\nu_t^P(g)
	= \sup_{P\in\mathrm{Proj}(W)} \mathbb{E}\left[
		e^{\varepsilon\|g_t^P\|_\mathfrak{g}^2/t}\right] <\infty
\]
and
\[ \int_G e^{\varepsilon\|g\|_\mathfrak{g}^2/t}\,d\nu_t(g)
	=\mathbb{E}\left[
    e^{\varepsilon\|g_t\|_{\mathfrak{g}}^2/t}\right] <\infty. \]
\end{prop}

The next proposition follows from Propositions \ref{p.bmapprox} and \ref{p.rho} and the proof
of Proposition 4.12 in \cite{DG08-3}.
\begin{prop}
\label{p.rho2}
Let
$\delta>0$ be as in Proposition \ref{p.rho}, and suppose that
$f:G\rightarrow\mathbb{C}$ is a continuous function such
that, for some $\varepsilon\in(0,\delta)$ and $p\in[1,\infty)$,
\[ |f(g)| \le C e^{\varepsilon \|g\|_\mathfrak{g}^2/pt}, \]
for all $g\in G$.  Then $f\in L^p(\nu_t)$, and, for all $h\in G$,
\begin{equation}
\label{e.h1}
\lim_{n\rightarrow\infty} \mathbb{E}|f(hg^n_t) - f(hg_t)|^p = 0
\end{equation}
and
\begin{equation}
\label{e.h2}
\lim_{n\rightarrow\infty} \mathbb{E}|f(g^n_th) - f(g_th)|^p = 0.
\end{equation}
\end{prop}

Finally, we include the following proposition, which states that, as the name
suggests, the Cameron-Martin subgroup is a subspace of heat kernel measure 0.
The proof is identical to Proposition 4.6 of \cite{DG08-3}.

\begin{prop}
\label{p.CM0}
For all $t>0$, $\nu_t(G_{CM})=0$.
\end{prop}

\begin{proof}
Let $\mu_t$ denote Wiener measure on $W$ with variance $t$.  Then for a
bounded measurable function $f:G=W\times\mathbf{C}\rightarrow\mathbb{C}$ such that
$f(w,x)=f(w)$,
\[ \int_G f(w)\,d\nu_t(w,x) = \mathbb{E}[f(B_t)] = \int_W f(w)\,d\mu_t(w). \]
Let $\pi:W\times\mathbf{C}\rightarrow W$ be the projection $\pi(w,x)=w$.  Then
$\pi_*\nu_t=\mu_t$, and thus
\[ \nu_t(G_{CM}) = \nu_t(\pi^{-1}(H)) = \pi_*\nu_t(H) = \mu_t(H)=0.\]
\end{proof}

\subsection{Holomorphic functions on $G$ and $G_{CM}$}
\label{s.h5.2}
We recall here the basic facts for holomorphic functions on
infinite-dimensional spaces required for the sequel.
For complete proofs of any of these results, see Section 5 of \cite{DG08-3}.

\subsubsection{Holomorphic functions on Banach spaces}
\label{s.h5.1}

The material in this subsection is based on the theory in \cite{HP74}.
Let $X$ and $Y$ be two complex Banach spaces, and for $a\in X$ and $\delta>0$
let
\[
B_{X}(a,\delta)
	:= \left\{  x\in X:\left\| x-a\right\|_X < \delta \right\}
\]
be the open ball in $X$ with center $a$ and radius $\delta$.  The following
is Definition 3.17.2 of Hille and Phillips \cite{HP74}.
\begin{defn}
\label{d.h5.1}
Let $\mathcal{D}$ be an open subset of $X$. A function $f:\mathcal{D}
\rightarrow Y\ $is said to be {\em holomorphic} or {\em analytic} if the
following two conditions hold.

\begin{enumerate}
\item $f$ is locally bounded, namely, for all $a\in\mathcal{D}$ there exists
$r_a>0$ such that
\[
M_a:=\sup\left\{ \|f(x)\|_Y: x\in B_{X}(a,r_a)  \right\}  <\infty.
\]

\item The function $f$ is complex G\^{a}teaux differentiable on
$\mathcal{D}$,  that is, for each $a\in\mathcal{D}$ and $h\in X$,  the function
$\lambda\mapsto f(a+\lambda h)$ is complex
differentiable at $\lambda=0 \in\mathbb{C}$.
\end{enumerate}
\begin{remark}
Holomorphic and analytic will be considered to be synonymous for the
purposes of this paper.  We will use ``holomorphic.''
\end{remark}
\end{defn}

The next proposition gathers together a number of basic properties of
holomorphic functions which may be found in \cite{HP74}, see also
\cite{Herve89}. One of the key ingredients to all of these results
is Hartog's theorem, see \cite[Theorem 3.15.1]{HP74}.


\begin{prop}
\label{t.h5.2}
If $f:\mathcal{D}\rightarrow Y$ is holomorphic, then
there exists a function
$f':\mathcal{D}\rightarrow\operatorname{Hom}\left(
X,Y\right)$, the space of bounded complex linear operators
from $X$ to $Y$, satisfying the following:

\begin{enumerate}
\item If $a\in\mathcal{D}$,  $x\in B_{X}\left(  a,r_{a}/2\right)$,  and $h\in
B_{X}\left(  0,r_{a}/2\right)$,  then
\[
\| f(x+h) - f(x) - f'(x)h \|_Y
	\le \frac{4M_a}{r_a(r_a - 2\|h\|_X)} \|h\|_X^2.
\]
In particular, $f$ is continuous and Frech\'{e}t differentiable on
$\mathcal{D}$.
\item The function $f':\mathcal{D}\rightarrow\operatorname{Hom}\left(
X,Y\right)$ is holomorphic.
\end{enumerate}
\end{prop}

By applying Proposition \ref{t.h5.2} repeatedly, it
follows that any holomorphic function $f:\mathcal{D}\rightarrow Y$
is Frech\'{e}t differentiable to all orders and each of the
Frech\'{e}t differentials is again a holomorphic function on
$\mathcal{D}$.

\subsubsection{Holomorphic functions on $G$ and $G_{CM}$}
Now we describe results for holomorphic functions on $G$ and $G_{CM}$.
For the next proposition, take $G_0=G$ and $\mathfrak{g}_0=\mathfrak{g}$
or $G_0=G_{CM}$ and $\mathfrak{g}_0=\mathfrak{g}_{CM}$. Note that
as usual we treat group elements as Lie algebra elements when we
write the group multiplication below. This linearization explains
why the proof is identical to \cite{DG08-3}, and why we omit it.

\begin{prop}
\label{l.h5.4}
For each $g\in G_0$, the left translation map $L_g:G_0\rightarrow
G_0$ is holomorphic in the $\|\cdot\|_{\mathfrak{g}_0} $-topology.
Moreover, a function $f:G_0\rightarrow\mathbb{C}$ defined in a
neighborhood of $g\in G_0$ is G\^{a}teaux (Frech\'{e}t) differentiable
at $g$ if and only if $f\circ L_g$ is G\^{a}teaux (Frech\'{e}t)
differentiable at $e$.  If $f$ is Frech\'{e}t differentiable at
$g$,  then
\[
(f\circ\,L_g)'(e) h
	= f'(g) \left(  h+\frac{1}{2}[g,h]  \right).
\]

Thus, a function $f:G_0\rightarrow\mathbb{C}$ is
holomorphic if and only if $f$ is locally bounded and $h\mapsto
f(g\cdot e^h)=f(g\cdot h)$ is G\^{a}teaux
(Frech\'{e}t) differentiable at $0$ for all $g\in G_0$.
If $f$ is holomorphic and $h\in\mathfrak{g}_0$,  then
\[
(\tilde{h}f)(g)
	= \frac{d}{d\lambda}\bigg|_0 f(g\cdot e^{\lambda h})
	= f'(g)\left(  h+\frac{1}{2}[g,h]  \right)
\]
is holomorphic as well.
\end{prop}

A simple induction argument using Proposition \ref{l.h5.4}
allows us to conclude that $\tilde{h}_{1}\dots\tilde{h}_{n}
f\in\mathcal{H}\left(  G_0\right)$ for all $f\in\mathcal{H}\left(
G_0\right)$ and $h_{1},\dots,h_{n}\in\mathfrak{g}_0$.

\begin{notation}
The space of globally defined holomorphic functions on a group $U$ will
be denoted by $\mathcal{H}(U)$.
\end{notation}

Finally, we also record the following result, which is completely analogous to
Proposition 5.7 and Corollary 5.8 of \cite{DG08-3}.

\begin{prop}
\label{p.h5.7}
If $f\in\mathcal{H}(G)$ and
$h\in\mathfrak{g}$,  then $\widetilde{ih}f=i\tilde{h}f$,
$\widetilde{ih}\bar{f}=-i\tilde{h}\bar {f}$,
\begin{align*}
\left(  \widetilde{ih}^2+\tilde{h}^2\right) f
	&= 0, \text{ and} \\
\left( \widetilde{ih}^2 + \tilde{h}^2\right) |f|^2
	&= 4 |\tilde{h}f|^2.
\end{align*}
Thus, for $L$ as in Proposition \ref{p.L} and $f:G\rightarrow\mathbb{C}$
a holomorphic cylinder function, $Lf=0$ and
\[
L|f|^2 = \sum_{j=1}^\infty \left| \tilde{\eta}_jf\right|^2
\]
for any $\{\eta_j\}_{j=1}^\infty$
a basis of $H\times\{0\}$ as in Notation \ref{n.pn}.
\end{prop}

\section{The Taylor isomorphism}
\label{s.Taylor}

Before we define the Taylor map, we must first define the relevant Hilbert
spaces.  First of these is the noncommutative Fock space, which plays the role
of the derivative space of holomorphic functions.

\subsection{Noncommutative Fock space}
\label{s.NCF}

We set the now standard notation for the noncommutative Fock
space, making the appropriate changes in the definition of the norm to
accomodate the subelliptic setting.

\begin{notation}
Let $V$ be a complex vector space.  We will denote the algebraic
dual to $V$ by $V'$.  For $k\in\mathbb{N}$, let $V^{\otimes k}$
denote the $k$-fold algebraic tensor product of $V$ with itself.
For any tensors $a,b$, we write $a\wedge b$ for $a\otimes b -
b\otimes a$.  Let $T(V)$ denote the algebraic tensor algebra over
$V$, so that $a\in T(V)$ is a finite sum
\[ a= \sum_{k=0}^n a_k, \qquad a_k\in V^{\otimes k}, \]
where $V^{\otimes0}=\mathbb{C}$.  For $\alpha\in T(V)'$ and
$k\in\{0\}\cup\mathbb{N}$, let $\alpha_k :=
    \alpha|_{V^{\otimes k}} \in \left(V^{\otimes k}\right)'$, so that
\[ \alpha = \sum_{k=0}^\infty \alpha_k, \qquad \alpha_k\in (V^{\otimes k})'. \]
When $V$ is a Lie algebra, let $J(V)$ be the two-sided ideal in $T(V)$ generated by
$\{a\wedge b -[a,b]:a,b\in V\}$ and let $J^0(V)$ be the backward
annihilator of $J(V)$, that is,
\[ J^0(V) = \{\alpha\in T(V)' : \langle\alpha,J(V)\rangle=0 \}. \]
\end{notation}

In particular, we will be concerned with the vector spaces $\mathfrak{g}_{CM}$
and $\mathfrak{g}_P=PW\times\mathbf{C}$.  We will let
$J^0(\mathfrak{g}_{CM})=J^0$.
Now we will define norms on $J^0$ and $J^0(\mathfrak{g}_P)$.

In order to put a norm on
$J^0$, let $\{\xi_j\}_{j=1}^\infty\subset H_*$ be a fixed complex orthonormal basis
of $H$ and $\{\eta_j\}_{j=1}^\infty=\{(\xi_j,0)\}_{j=1}^\infty$ be a
complex basis of $H\times\{0\}$ as in Notation \ref{n.pn}.
For $k\in\{0\}\cup\mathbb{N}$, we define a non-negative sesqui-linear form on
$(\mathfrak{g}_{CM}^{\otimes k})'$ by
\[ (\alpha,\beta)_k := \sum_{j_1,\dots,j_k=1}^\infty
	\langle \alpha, \eta_{j_1}\otimes\cdots\otimes \eta_{j_k}\rangle
	\overline{\langle \beta, \eta_{j_1}\otimes\cdots\otimes
	\eta_{j_k}\rangle},
\text{ for all } \alpha, \beta\in (\mathfrak{g}_{CM}^{\otimes k})'. \]
For $\alpha\in(\mathfrak{g}_{CM}^{\otimes k})'$, we will write
\[ \|\alpha\|_k^2 := (\alpha,\alpha)_k = \sum_{j_1,\dots,j_k=1}^\infty
    |\langle \alpha, \eta_{j_1}\otimes\cdots\otimes \eta_{j_k}\rangle|^2. \]
The following lemma is clear from the definition of $\|\cdot\|_k$.
\begin{lem}
Let $\alpha\in(\mathfrak{g}_{CM}^{\otimes k})'$ for some $k\in\mathbb{N}$.
Then $\|\alpha\|_k>0$ if and only if there exist some $\xi_1,\ldots,\xi_k\in H$
such that $\langle \alpha,(\xi_1,0)\otimes \cdots \otimes(\xi_k,0)\rangle\ne
0$.
\end{lem}

For any projection $P\in\mathrm{Proj}(W)$,
we define an analogous norm for the finite-dimensional
Lie algebras $\mathfrak{g}_P= PW\times\mathbf{C}$.
Let $\{\xi_j\}_{j=1}^n$ be a complex orthonormal basis for $PH$, and let
$\{\eta_j\}_{j=1}^n=\{(\xi_j,0)\}_{j=1}^n$.
Define the non-negative sesqui-linear form
\[ (\alpha,\beta)_P := \sum_{j=1}^n \langle \alpha,\eta_j\rangle
    \overline{\langle \beta,\eta_j\rangle} \qquad \text{ for all }
    \alpha,\beta\in\mathfrak{g}_P'. \]
This induces a form on $(\mathfrak{g}_P^{\otimes k})'$ determined by
\[
(\alpha_1\otimes\cdots\otimes\alpha_k,\beta_1\otimes\cdots\otimes\beta_k)_{P,k}
    := \prod_{\ell=1}^k (\alpha_\ell,\beta_\ell)_P \qquad \text{ for all }
    \alpha_j,\beta_j\in\mathfrak{g}_P'\]
For $\alpha\in(\mathfrak{g}_P^{\otimes k})'$, we will write
\[ \|\alpha\|_{P,k}^2 := (\alpha,\alpha)_{P,k}
	= \sum_{j_1,\dots,j_k=1}^n |\langle \alpha,\eta_{j_1}\otimes\cdots\otimes
        \eta_{j_k}\rangle|^2. \]
One may easily verify that $\|\cdot\|_k$ and $\|\cdot\|_{P,k}$ are
independent of the choice of orthonormal basis.

\begin{defn}
[Noncommutative Fock spaces]
\label{d.NCF}
For $t>0$ and $\alpha=\sum_{k=1}^\infty\alpha_k\in J^0$, let
\[
\|\alpha\|_t^2 := \sum_{k=0}^\infty \frac{t^k}{k!}\|\alpha_k\|_k^2,
\]
and
\[ J_t^0 := \{ \alpha\in J^0: \|\alpha\|_t<\infty \}. \]
Similarly, for $t>0$, $P\in\mathrm{Proj}(W)$, and $\alpha\in
J^0(\mathfrak{g}_P)$, let
\[  \|\alpha\|_{P,t}^2 := \sum_{k=0}^\infty \frac{t^k}{k!}
	\|\alpha_k\|_{P,k}^2, \]
and
\[ J_{P,t}^0 := \{ \alpha\in J^0(\mathfrak{g}_P): \|\alpha\|_{P,t}<\infty \}. \]
For $\{P_n\}_{n=1}^\infty$ an increasing sequence of projections in
$\mathrm{Proj}(W)$,
let $\|\cdot\|_{n,k}:=\|\cdot\|_{P_n,k}$,
$\|\alpha\|_{n,t}:=\|\alpha\|_{P_n,t}$, $J_{n,t}^0 := J_{P_n,t}^0$.
\end{defn}

The functions $\|\cdot\|_t$ and $\|\cdot\|_{P,t}$ are clearly semi-norms on
$J_t^0$ and $J_{P,t}^0$, respectively.  It is proved in Theorem 2.7 of \cite{DGSC09-2} that, for any $t>0$ and
$P\in\mathrm{Proj}(W)$ ,
the semi-norm $\|\cdot\|_{P,t}$ is a norm on $J_{P,t}^0$ (using the fact that
$[PW,PW]=\mathbf{C}$).  In fact, $J_{P,t}^0$ is a
Hilbert space when equipped with the inner product
\[ \langle \alpha, \beta\rangle_{P,t}
    := \sum_{k=0}^\infty \frac{t^k}{k!} (\alpha_k,\beta_k)_{P,k}
    \qquad \text{ for all } \alpha, \beta\in J_{P,t}^0. \]

To compare our notation with that used in \cite{DGSC09-2},
for each $P\in\mathrm{Proj}(W)$, let
\[ K_P := \left\{\alpha\in \mathfrak{g}_P' :
    (\alpha,\alpha)_P = \sum_{j=1}^n |\langle\alpha,\eta_j\rangle|^2
    = 0 \right\}. \]
Then clearly
\[ K_P^0 := \{a\in\mathfrak{g}_P : \langle \alpha,a\rangle = 0 \text{ for all
    } \alpha\in K_P \} = PH\times\{0\}. \]
If the Lie algebra generated by $PH$ is all of
$\mathfrak{g}_P$, then $(\cdot,\cdot)_P$ satisfies H\"ormander's condition
as defined in Definition 2.6 of \cite{DGSC09-2}.

Here we follow the proof in \cite{DGSC09-2} to show that, since
H\"ormander's condition $[H,H]=\mathbf{C}$ holds,
$\|\cdot\|_t$ is a norm on $J_t^0$.  (Indeed, it is shown in \cite{DGSC09-2}
that, at least in the finite-dimensional case, $\|\cdot\|_t$ is a norm on $J_t^0$
if and only if H\"ormander condition holds.)  First, we need the following lemma.

\begin{lem}
\label{l.P}
There exists an algebra homomorphism $\Psi: T(\mathfrak{g}_{CM})\rightarrow
T(H)$ such that $ T(\mathfrak{g}_{CM})= T(H) \oplus \mathrm{Nul}(\Psi)$, where
$\mathrm{Nul}(\Psi)\subset J(\mathfrak{g}_{CM})$.
\end{lem}
\begin{proof}
Let $\{\xi_j\}_{j=1}^\infty$ be an orthonormal basis of $H$.  Since
$[H,H]=\mathbf{C}$, we may also choose
$\{A_\ell,B_\ell\}_{\ell=1}^N\subset H$ such that
$\{\omega(A_\ell,B_\ell)\}_{\ell=1}^N$ is a basis of $\mathbf{C}$ with dual
basis $\{\varepsilon^\ell\}_{\ell=1}^N$.
Define $\psi:\mathfrak{g}_{CM}\rightarrow H\oplus H^{\otimes2}$
for
\[ (A,a) = \sum_{j=1}^\infty \langle A,\xi_j\rangle_H (\xi_j,0)
		+ \sum_{\ell=1}^N \varepsilon^\ell(a) (0,\omega(A_\ell,B_\ell))
		\in\mathfrak{g}_{CM} \]
by
\[ \psi(A,a) := \sum_{j=1}^\infty  \langle A,\xi_j\rangle_H (\xi_j,0)
		+ \sum_{\ell=1}^N \varepsilon^\ell(a) (A_\ell\wedge B_\ell,0),
\]
where again $u\wedge v = u\otimes v -v\otimes u$ for any $u,v\in H$.
Then $\psi$ is a linear operator such that $\psi (A,0)=(A,0)$ for any $A\in H$, and, as
\[ (A\wedge B,0) - (0,\omega(A,B)) = (A,0)\wedge (B,0)-(0,\omega(A,B))
		\in J(\mathfrak{g}_{CM}),
\]
for any $A,B\in H$, we have
$\psi h-h\in J(\mathfrak{g}_{CM})$ for all $h\in\mathfrak{g}_{CM}$.  One may also show that
$\psi$
is bounded as an operator into $T(H)$: for any $x=(A,a)\in G_{CM}$ such that
$\|x\|_{\mathfrak{g}_{CM}}^2=\|A\|_H^2+\|a\|_\mathbf{C}\le 1$,
\begin{align*}
\|\psi(A,a)\|_{H\oplus H^{\otimes 2}}^2
	&= \|A\|_H^2 + \sum_{j,k=1}^\infty
		\left|\left\langle \sum_{\ell=1}^N \varepsilon^\ell(a) A_\ell\wedge
		B_\ell, \xi_j\otimes \xi_k\right\rangle\right|^2 \\
	&\le \|A\|_H^2 + \sum_{j,k=1}^\infty
		\left(\sum_{\ell=1}^N \varepsilon^\ell(a)^2
		\sum_{\ell=1}^N|\langle A_\ell\wedge
		B_\ell, \xi_j\otimes \xi_k \rangle|^2 \right)\\
	&\le \|A\|_H^2 + C\|a\|_\mathbf{C}^2
	\le C'(\|A\|_H^2+\|a\|_\mathbf{C}),
\end{align*}
where $C'=C'(N,\omega)<\infty$, and the final inequality follows from the fact
that $\|A\|_H^2+\|a\|_\mathbf{C}\le1$ implies that
$\|A\|_H^2+\|a\|_\mathbf{C}^2\le\|A\|_H^2+\|a\|_\mathbf{C}$.

By the universal property of the tensor algebra, there is a unique extension
of $\psi$ to an algebra homomorphism $\Psi: T(\mathfrak{g}_{CM})\rightarrow T(H)$,
such that $\Psi 1_{T(\mathfrak{g}_{CM})}=1_{T(H)}$.  Since for
$h_1,\dots,h_n\in\mathfrak{g}_{CM}$
\[ \Psi (h_1\otimes\cdots\otimes h_n) = \psi h_1\otimes\cdots\otimes \psi h_n
	\in (h_1+J(\mathfrak{g}_{CM}))\otimes \cdots \otimes (h_n+J(\mathfrak{g}_{CM}))\]
and $J(\mathfrak{g}_{CM})$ is an ideal, it follows that $\Psi (h_1\otimes\cdots\otimes
h_n)-h_1\otimes\cdots\otimes h_n\in J(\mathfrak{g}_{CM})$.
\end{proof}

This lemma immediately gives the following.

\begin{thm}
Let $t>0$.  The semi-norm $\|\cdot\|_t$ on $J_t^0$ is a norm.
\end{thm}

\begin{proof}
Suppose that $\alpha=\sum_{k=0}^\infty\alpha_k\in J^0$
is such that
\[ 0 = \|\alpha\|_t^2
	= \sum_{k=0}^\infty \frac{t^k}{k!}
		\sum_{i_1,\ldots,i_k=1}^\infty |\langle \alpha_k,
		\eta_{i_1}\otimes\cdots\otimes \eta_{i_k}\rangle|^2. \]
Thus, $\alpha|_{T(H)}=0$ and, for $\Psi$ as in Lemma \ref{l.P},
$\alpha=\alpha\circ \Psi = \alpha|_{T(H)}\circ \Psi=0$.
\end{proof}

\begin{cor}
The space $J_t^0$ is a Hilbert space equipped with the inner product
\[ \langle\alpha,\beta\rangle_t
	:=\sum_{k=0}^\infty \frac{t^k}{k!}(\alpha_k,\beta_k)_k. \]
\end{cor}

\subsection{The Taylor map}

The other relevant space for the Taylor map should be thought of as the $\nu_t$-square
integrable holomorphic functions on $G_{CM}$.  For $t>0$, $f:G_{CM}\rightarrow\mathbb{C}$, and $P\in\mathrm{Proj}(W)$, let
\[ \|f\|^2_{L^2(\nu_t^P)} := \|f|_{G_P}\|^2_{L^2(\nu_t^P)}
	= \mathbb{E}|f(g_t^P)|^2, \]
where $\{g_t^P\}_{t\ge0}\subset G_P\subset G_{CM}$
is a Brownian motion on $G_P$ as in Proposition \ref{p.bmapprox}.

\begin{defn}
For $t>0$ and $f\in\mathcal{H}(G_{CM})$, let
\[ \|f\|_{\mathcal{H}^2_t(G_{CM})}
    := \sup _{P\in\mathrm{Proj}(W)} \|f \|_{L^2(\nu_t^P)},
\]
and define
\[ \mathcal{H}^2_t(G_{CM}) := \{ f\in \mathcal{H}(G_{CM}):
    \|f\|_{\mathcal{H}^2_t(G_{CM})} <\infty \}. \]
\end{defn}

We set one more piece of notation before defining the Taylor map.

\begin{notation}
\label{n.fhat}
Given $f\in\mathcal{H}(G_{CM})$, $g\in G_{CM}$, $k\in\{0\}\cup\mathbb{N}$, let
$\hat{f}_k(g):=(D^kf)(g)$ denote the
unique element of $(\mathfrak{g}_{CM}^{\otimes k})'$ given by
\begin{align*}
(D^0f)(g)&=f(g) \\
\langle (D^kf)(g),h_1\otimes\cdots\otimes h_k\rangle
    &= \left(\tilde{h}_1\cdots\tilde{h}_k f\right)(g)
\end{align*}
for all $h_1,\dots,h_k\in\mathfrak{g}_{CM}$. Let $\hat{f}(g)$ be the element
of $ T(\mathfrak{g}_{CM})'$ determined by
\[ \langle \hat{f}(g),\beta\rangle = \langle \hat{f}_k(g),\beta\rangle,
    \qquad \text{ for all } \beta\in\mathfrak{g}_{CM}^{\otimes k}. \]
\end{notation}

\begin{remark}
As a consequence of equation (\ref{e.hk}), $\hat{f}(g)\in J^0$ for
all $f\in\mathcal{H}(G_{CM})$ and $g\in G_{CM}$.
\end{remark}

\begin{defn}
For each $t>0$, the {\em Taylor map} is the linear map
$\mathcal{T}_t:\mathcal{H}_t^2(G_{CM})\rightarrow J_t^0$ defined by
$\mathcal{T}_tf=\hat{f}(e)$.
\end{defn}

\subsection{Proof of isometry}

We will prove that the Taylor map is an isometry by limiting arguments for the
finite-dimensional projections.  Let us first recall the finite-dimensional
theory.

\begin{notation}
\label{n.ffhat}
For any $P\in\mathrm{Proj}(W)$, we set derivative notation for
$f\in\mathcal{H}(G_P)$ similarly to how it was done in Notation \ref{n.fhat}.
That is, for $g\in G_P$ and $k\in\{0\}\cup\mathbb{N}$, let
$\hat{f}_k(g):=(D_P^kf)(g)$ denote the
element of $(\mathfrak{g}_P^{\otimes k})'$ given by
\[
\langle (D_P^kf)(g),h_1\otimes\cdots\otimes h_k\rangle
    = \left(\tilde{h}_1\cdots\tilde{h}_k f\right)(g),
\]
for all $h_1,\dots,h_k\in\mathfrak{g}_P$,
and let $\hat{f}(g)$ be the element of $ T(\mathfrak{g}_P)'$ determined by
\[ \langle \hat{f}(g),\beta\rangle = \langle \hat{f}_k(g),\beta\rangle,
    \qquad \text{ for all } \beta\in\mathfrak{g}_P^{\otimes k}. \]

Also, let $\mathcal{H}L^2(\nu_t^P) = \mathcal{H}(G_P) \cap
L^2(G_P,\nu_t^P)$.  If $\{P_n\}_{n=1}^\infty$ is an increasing sequence in
$\mathrm{Proj}(W)$, let $\mathcal{H}L^2(\nu_t^n) = \mathcal{H}L^2(\nu_t^{P_n})$.
The {\em finite-dimensional Taylor map} is the
linear map $f \mapsto \hat{f}(e)$ from $\mathcal{H}L^2(\nu_t^P)$ to
$J^0_{P,t}$, where the latter is as defined in Definition \ref{d.NCF}
\end{notation}

For each $P\in\mathrm{Proj}(W)$, $G_P$ is a finite-dimensional connected,
simply connected complex Lie group.  If $[PW,PW]=\mathbf{C}$, then
$(\cdot,\cdot)_P$ is a non-negative Hermitian form on $\mathfrak{g}_P'$
satisfying H\"ormander's condition.  Thus, we have the following theorem.
\begin{thm}
\label{t.fdt}
Suppose that $P\in\mathrm{Proj}(W)$ such that
$[PW,PW]=\mathbf{C}$.  Then the finite-dimensional Taylor map $f \mapsto \hat{f}(e)$ is a unitary map
from $\mathcal{H}L^2(\nu_t^P)$ onto $J^0_{P,t}$.
Moreover, for any $t>0$, $f\in\mathcal{H}L^2(\nu_t^P)$, and $g\in G_P$,
\begin{equation}
\label{e.ptwise}
|f(g)|\le \|\hat{f}(e)\|_{P,t} e^{d_h^2(e,g)/2t}
\end{equation}
where $d_h$ is the horizontal distance on $G_P$ (defined analogously on $G_P$
to the horizontal distance on $G_{CM}$ as in Notation \ref{n.length}).
\end{thm}
The isometry and surjectivity follow from the finite-dimensional Taylor isomorphism
proved in Theorem 6.1 of \cite{DGSC09-2}, and the estimate in (\ref{e.ptwise})
is a consequence of Corollary 5.15 of that same reference.
The paper \cite{DGSC09-1} gives an alternate proof of the surjectivity,
as each $G_P$ is a nilpotent Lie group.  In Section \ref{s.Tsurj},
we will apply the methods used in \cite{DGSC09-1} to show that the
Taylor map is surjective in this infinite-dimensional setting as
well.  Here we use the finite-dimensional isometries to show that
$\mathcal{T}_t$ is an isometry for all $t>0$ as follows.

\begin{prop}
\label{p.Tiso}
Let $f\in \mathcal{H}(G_{CM})$ and $t>0$.  Then
\[ \|\hat{f}(e)\|_t
    = \|f\|_{\mathcal{H}^2_t(G_{CM})}. \]
\end{prop}
\begin{proof}
By the finite-dimensional Taylor isomorphism theorem, for all $P\in\mathrm{Proj}(W)$,
\[ \|\hat{f}(e)\|_{J^0_{P,t}}
    = \|f\|_{L^2(\nu_t^P)}. \]
Thus, by definition of $\|\cdot\|_{\mathcal{H}^2_t(G_{CM})}$,
\[ \|f\|_{\mathcal{H}^2_t(G_{CM})}
    = \sup_{P\in\mathrm{Proj}(W)} \|f\|_{L^2(\nu_t^P)}
    = \sup_{P\in\mathrm{Proj}(W)}
        \|\hat{f}(e)\|_{J^0_{P,t}}. \]
So showing that
\[ \sup_{P\in\mathrm{Proj}(W)}
        \|\hat{f}(e)\|_{J^0_{P,t}}
    = \|\hat{f}(e)\|_t \]
completes the proof.

Let $P\in\mathrm{Proj}(W)$ with $\{\xi_j\}_{j=1}^\infty$ an
orthonormal basis of $H$, such that $\{\xi_j\}_{j=1}^n$ is an
orthonormal basis of $PH$.  Let $\eta_j=(\xi_j,0)$.  Then
\begin{align*}
\|\hat{f}(e)\|_{J^0_t(\mathfrak{g}_P)}
    &= \sum_{k=0}^\infty \frac{t^k}{k!} \sum_{j_1,\ldots,j_k=1}^n
		|\langle \hat{f}(e),\eta_{j_1}\otimes\cdots\otimes
		\eta_{j_k}\rangle|^2 \\
    &\le \sum_{k=0}^\infty \frac{t^k}{k!} \sum_{j_1,\ldots,j_k=1}^\infty
		|\langle \hat{f}(e),\eta_{j_1}\otimes\cdots\otimes \eta_{j_k}\rangle|^2
    = \|\hat{f}(e)\|_t,
\end{align*}
and so $\sup_{P\in\mathrm{Proj}(W)} \|\hat{f}(e)\|_{J^0_t(\mathfrak{g}_P)}
\le \|\hat{f}(e)\|_t$.  On the other hand, if
$\{P_n\}_{n=1}^\infty\subset\mathrm{Proj}(W)$ is an
increasing sequence of projections, then
\begin{align*}
\sup_{P\in\mathrm{Proj}(W)}
        \|\hat{f}(e)\|_{J^0_t(\mathfrak{g}_P)}
    &\ge \lim_{n\rightarrow\infty} \|\hat{f}(e)\|_{n,t} \\
    &= \lim_{n\rightarrow\infty}
        \sum_{k=0}^\infty \frac{t^k}{k!} \sum_{j_1,\ldots,j_k=1}^n
        |\langle \hat{f}(e), \eta_{j_1}\otimes\cdots\otimes\eta_{j_k}\rangle|^2
        \\
    &= \sum_{k=0}^\infty \frac{t^k}{k!} \sum_{j_1,\ldots,j_k=1}^\infty
        |\langle \hat{f}(e), \eta_{j_1}\otimes\cdots\otimes\eta_{j_k}\rangle|^2
    = \|\hat{f}(e)\|_t.
\end{align*}
\end{proof}

The following corollary follows from Propositions \ref{p.Tiso} and \ref{p.rho2}.

\begin{cor}
\label{c.p}
Let $\delta>0$ be as in Proposition
\ref{p.rho}, and suppose that $f:G\rightarrow\mathbb{C}$ is a continuous
function such that $f|_{G_{CM}}\in\mathcal{H}(G_{CM})$ and,
for some $\varepsilon\in(0,\delta)$,
\[ |f(g)|\le Ce^{\varepsilon\|g\|_\mathfrak{g}^2/2t} \]
for all $g\in G$.  Then $f|_{G_{CM}}\in\mathcal{H}_t^2(G_{CM})$ and
$\widehat{f|}_{G_{CM}}(e)\in J_t^0$.
\end{cor}

In particular, Corollary \ref{c.p} implies that, for all $t>0$,
$\mathcal{P}_{CM}\subset\mathcal{H}_t^2(G_{CM})$ and, for any
$p\in\mathcal{P}$, $\widehat{p|}_{G_{CM}}(e)\in J_t^0$.  Thus,
$\mathcal{H}_t^2(G_{CM})$ and $J_t^0$ are non-trivial spaces.

\begin{cor}
The Taylor map $\mathcal{T}_t:\mathcal{H}_t^2(G_{CM})\rightarrow J_t^0$
is injective, and $\|\cdot\|_{\mathcal{H}^2_t(G_{CM})}$ is a norm on
$\mathcal{H}^2_t(G_{CM})$ induced by the inner product
\[ \langle u,v\rangle_{\mathcal{H}_t^2(G_{CM})} :=
    \langle\hat{u}(e),\hat{v}(e)\rangle_t,
    \qquad \text{ for all } u,v\in\mathcal{H}_t^2(G_{CM}). \]
\end{cor}

\begin{proof}
If $\hat{f}(e)=0$, then Proposition \ref{p.Tiso} implies that
$\|f\|_{\mathcal{H}_t^2(G_{CM})}=0$ and thus $f|_{G_P}=0$ for all
$P\in\mathrm{Proj}(W)$.  As $f$ is continuous and $\cup_{P\in\mathrm{Proj}(W)}
G_P$ is dense in $G_{CM}$ by Proposition \ref{p.equivtop},
it follows that $f\equiv0$.  Thus, $\mathcal{T}_t$ is injective.

Since $\|\cdot\|_t$ is a Hilbert norm, Proposition \ref{p.Tiso} then also
implies that $\|\cdot\|_{\mathcal{H}^2_t(G_{CM})}$ is the norm on
$\mathcal{H}^2_t(G_{CM})$ given by the above inner product.
\end{proof}

\subsection{A density theorem and proof of surjectivity}
\label{s.Tsurj}
We will now apply the methods used in
\cite{DGSC09-1} to show that the Taylor map is surjective.  In fact, the
infinite-dimensional proof is directly analogous to the finite-dimensional
proof presented there, and no special considerations need to be made for the infinite-dimensional case.
Similar arguments were used in \cite{Cecil08} and \cite{DG08-3}.
Still, we collect the proofs here for completeness
and to stress the dimension independence of the arguments.  Additionally,
Corollary \ref{c.h7.4} will be critical in the proof of surjectivity of the
restriction map in Section \ref{s.restriction}, and this proof will require
some adaptation for the subelliptic construction.

\begin{defn}
\label{d.h7.2}
A tensor $\alpha = \sum_{k=0}^\infty \alpha_k \in T(\mathfrak{g}_{CM})'$
is said to have {\em finite rank} if $\alpha_k=0$ for all but
finitely many $k\in\mathbb{N}$.
\end{defn}

The next lemma is essentially a special case of \cite[Lemma 3.5]{DGSC09-2}.
See also \cite[Theorem 41]{Cecil08} and \cite[Lemma 7.3]{DG08-3}.

\begin{lem}
\label{l.h7.3}The finite rank tensors in $J_t^0$ are dense in $J_t^0$.
\end{lem}

\begin{proof}
First note that $\mathfrak{g}_{CM}=H\times \mathbf{C}$ is a graded Lie algebra with
$[H,H]=\mathbf{C}$, $[H,\mathbf{C}]=0$, and $[\mathbf{C},\mathbf{C}]=0$.
Thus, for $\theta\in\mathbb{R}$, we may define the dilations
$\varphi_\theta:\mathfrak{g}_{CM}\rightarrow\mathfrak{g}_{CM}$ by
\[ \varphi_\theta(A,a) := (e^{i\theta}A, e^{2i\theta}a),
    \qquad \text{ for all } (A,a)\in\mathfrak{g}_{CM}, \]
and it is straightforward to verify that $\varphi_\theta$ is an automorphism of
$\mathfrak{g}_{CM}$.
Let $\Phi_\theta :  T(\mathfrak{g}_{CM})\rightarrow T(\mathfrak{g}_{CM})$ be the
automorphism of the tensor algebra over $\mathfrak{g}_{CM}$ induced by
$\varphi_\theta$, that is,
\[ \Phi_\theta :=
    \overset{k \text{ times}}
    {\overbrace{\varphi_\theta\otimes\cdots\otimes\varphi_\theta}}
    \text{ on } \mathfrak{g}_{CM}^{\otimes k}. \]
Then
\begin{align*}
\Phi_\theta (\xi\wedge\xi'-[\xi,\xi'])
	&=(\varphi_\theta\xi)\wedge(\varphi_\theta\xi')-\varphi_\theta[\xi,\xi']\\
	&=(\varphi_\theta\xi)\wedge(\varphi_\theta\xi')-[\varphi
_\theta\xi,\varphi_\theta\xi'].
\end{align*}
From this it follows that $\Phi_\theta(J)\subset J$
and therefore if $\alpha\in J^0$,
then $\alpha\circ\Phi_\theta \in J^0$. Letting $\{\xi_j\}_{j=1}^\infty$ be an
orthonormal basis of $H$ and $\Gamma=\{(\xi_j,0)\}_{j=1}^\infty$,
we have $\varphi_\theta \eta=e^{i\theta}\eta$
for all $\eta\in\Gamma$. Therefore,
\begin{align*}
|\langle\alpha\circ\Phi_\theta ,\eta_1\otimes\dots\otimes \eta_k\rangle| ^2
	&= |\langle \alpha,
		\varphi_\theta \eta_1\otimes\dots\otimes\varphi_\theta 	
		\eta_k\rangle|^2 \\
	&= |\langle\alpha, \eta_1\otimes\dots\otimes \eta_k\rangle|^2,
\end{align*}
and hence
\begin{align*}
\| \alpha\circ\Phi_\theta \|_t^2
	&= \sum_{k=0}^\infty\frac{t^k}{k!} \sum_{\eta_1,\dots,\eta_k\in\Gamma}
		|\langle\alpha\circ \Phi_\theta ,
		\eta_1\otimes\dots\otimes\eta_k\rangle|^2 \\
	&=\sum_{k=0}^\infty\frac{t^k}{k!}\sum_{\eta_1,\dots,\eta_k\in\Gamma}
		|\langle\alpha, \eta_1\otimes\dots\otimes\eta_k\rangle| ^2
	= \left\|\alpha\right\|_t^2.
\end{align*}
So the map $J_t^0\ni\alpha\mapsto
\alpha\circ\Phi_\theta \in J_t^0$ is unitary. Moreover, since
\[ | \langle\alpha,\varphi_\theta \eta_1\otimes
		\dots\otimes\varphi_\theta \eta_k\rangle
		- \langle\alpha, \eta_1\otimes\dots\otimes\eta_k\rangle|^2 \\
	\le 2 | \langle\alpha, \eta_1\otimes\dots\otimes\eta_k\rangle|^2,
\]
the dominated convergence theorem implies that
\begin{equation}
\label{e.Phi}
\begin{split}
\lim_{\theta\rightarrow0} &
		\left\| \alpha\circ\Phi_\theta -\alpha\right\|_t^2\\
	&= \sum_{k=0}^\infty\frac{t^k}{k!}\sum_{\eta_1,\dots,\eta_k\in\Gamma}
		\lim_{\theta\rightarrow0} \left|\langle \alpha, \varphi_\theta
		\eta_1\otimes \dots\otimes\varphi_\theta
		\eta_k\rangle-\langle\alpha, \eta_1\otimes \dots\otimes\eta_k
		\rangle\right| ^2\\
	&= 0,
\end{split}
\end{equation}
and $\alpha\mapsto\alpha\circ\Phi_\theta$ is continuous.
(Notice that $\Phi_\theta \circ\Phi_{\alpha}=\Phi_{\theta+\alpha}$,
so it suffices to check continuity at $\theta=0$.)

Now, for any $n\in\mathbb{N}$, let
\[
F_n(\theta)
	= \frac{1}{2\pi n} \sum_{j=0}^{n-1} \sum_{\ell=-j}^j
		e^{i\ell\theta}
	= \frac{1}{2\pi n} \frac{\sin^2(j\theta/2)}{\sin^2(\theta/2)}
\]
denote Fejer's kernel \cite[p. 143]{T}. Then one may show the following:
$\int_{-\pi}^{\pi}F_n(\theta)d\theta=1$ for all $n\in\mathbb{N}$;
\[
\lim_{n\rightarrow\infty}\int_{-\pi}^{\pi}F_n(\theta)u(\theta)d\theta
	= u(0),
\]
for all continuous functions $u:[-\pi,\pi]\rightarrow\mathbb{C}$; and
\[ \int_{-\pi}^\pi F_n(\theta) e^{im\theta}\,d\theta = 0 \]
whenever $m>n$.  Given $\alpha\in J_t^0$, we let
\[
\alpha(n)
	:= \int_{-\pi}^\pi \alpha\circ\Phi_\theta F_n(\theta) \,d\theta.
\]
If $\beta=h_1\otimes\cdots\otimes h_m\in\mathfrak{g}^{\otimes
m}_{CM}$,  then
there exist $\beta_\ell\in\mathfrak{g}_{CM}^{\otimes m}$ such that
\[
\Phi_\theta \beta=\sum_{\ell=m}^{2m}e^{i\ell\theta}\beta_\ell.
\]
So, if $m>n$,
\[
\left\langle \alpha(n), \beta\right\rangle
	= \int_{-\pi}^\pi \left\langle\alpha,\Phi_\theta \beta\right\rangle
		F_n(\theta) \, d\theta
	= \sum_{\ell=m}^{2m}\left\langle \alpha, \beta_\ell\right\rangle
		\int_{-\pi}^\pi e^{i\ell\theta}F_n(\theta)\,d\theta=0,
\]
from which it follows that $\alpha(n)_m\equiv0$ for
all $m>n$. Thus $\alpha(n)  $ is a finite rank tensor
for all $n\in\mathbb{N}$, and (\ref{e.Phi}) implies that
\begin{align*}
\limsup_{n\rightarrow\infty} \|\alpha-\alpha(n)\|_t^2
	&= \limsup_{n\rightarrow\infty}
		\left\| \int_{-\pi}^{\pi} [ \alpha-\alpha\circ\Phi_\theta ]
		F_n(\theta)\,d\theta\right\|_t \\
	&\le \limsup_{n\rightarrow\infty} \int_{-\pi}^\pi
		\|\alpha-\alpha\circ\Phi_\theta\|_t F_n(\theta)\,d\theta
	= 0.
\end{align*}

\end{proof}

The surjectivity of the Taylor map may now be proved by finding a preimage
in $\mathcal{H}_t^2(G_{CM})$ under $\mathcal{T}_t$ for any finite rank tensor
in $J_t^0$.  The following lemma is a special case of Proposition 5.1 in
\cite{D95} and motivates our construction of the inverse of the Taylor map.
This version of the result may also be found in Lemma 6.9 of \cite{DG08-3}.

\begin{lem}
\label{l.holo}
For every $f\in\mathcal{H}(G_{CM})$ and $g\in G_{CM}$,
\[ f(g) = \sum_{k=0}^\infty \frac{1}{k!}
	\langle \hat{f}_k(e),g^{\otimes k}\rangle, \]
where by convention $g^{\otimes 0}=1\in\mathbb{C}$ and
the above sum is absolutely convergent.
\end{lem}

\begin{proof}
The function $u(z):=f(zg)$ is a holomorphic function of $z\in\mathbb{C}$.
Therefore,
\[ f(g) = u(1) = \sum_{k=0}^\infty \frac{1}{k!}u^{(k)}(0), \]
and the sum is absolutely convergent.  In fact, for all $r>0$, there exists
$C(r)<\infty$ such that $\frac{1}{k!}|u^{(k)}(0)|\le C(r)r^{-k}$ for all
$k\in\mathbb{N}$.  Finally, note that
\begin{align*}
u^{(k)}(0) &= \frac{d^k}{dt^k}\bigg|_{t=0} u(t)
	= \frac{d^k}{dt^k}\bigg|_{t=0} f(tg) \\
	&= \frac{d^k}{dt^k}\bigg|_{t=0} f(e^{tg})
	= (\tilde{g}^kf)(e)
	= \langle \hat{f}_k(e),g^{\otimes k}\rangle.
\end{align*}
\end{proof}

The following proof of the surjectivity of the Taylor map
is directly analogous to the proof of Lemma 3.6 in \cite{DGSC09-1}.

\begin{thm}
\label{t.surj}
The Taylor map $\mathcal{T}_t:\mathcal{H}_t^2(G_{CM})\rightarrow J_t^0$
is surjective.
\end{thm}
\begin{proof}
Consider first $\alpha$ a finite rank tensor in $J_t^0$.  By Lemma
\ref{l.holo}, if $f= \mathcal{T}_t^{-1}\alpha$ exists, then it must be given
by
\[ f_\alpha(g) := \sum_{k=0}^\infty \frac{1}{k!}
	\langle \alpha_k,g^{\otimes k}\rangle,  \]
for all $g\in G_{CM}$.  This is a finite sum since $\alpha$ is of finite rank,
and thus $f_\alpha$ is a finite sum of
continuous complex multilinear forms in $g\in G_{CM}$.  Thus, $f_\alpha$ is
holomorphic, and, in particular, for any $h\in\mathfrak{g}_{CM}$,
\[ \langle \hat{f}_\alpha(e),h^{\otimes k}\rangle
	= \frac{d^k}{dt^k}\bigg|_{t=0} f_\alpha(th)
	= \frac{d^k}{dt^k}\bigg|_{t=0} \sum_{n=0}^\infty \frac{1}{n!}
		\langle \alpha_n,(th)^{\otimes n}\rangle
	= \langle\alpha_k,h^{\otimes k}\rangle. \]
So $\hat{f}_\alpha(e)=\alpha$ on $\mathrm{span}\{h^{\otimes
k}:h\in\mathfrak{g}_{CM},k\in\{0\}\cup\mathbb{N}\} = \{\text{symmetric
} \mathbb{R}\text{-tensors}\}=:\mathcal{S}$.  By the Poincar\'e-Birkhoff-Witt
theorem (see \cite[Lemma 3.3.3]{Varadarajan84} or \cite[Corollary
E]{HumphreysBook}),
$T(\mathfrak{g}_{CM})=\mathcal{S}\oplus J$, and, since
$\hat{f}_\alpha(e)-\alpha$ annihilates $J$, this implies that
$\hat{f}_\alpha(e)=\alpha$ on $T(\mathfrak{g}_{CM})$.

Thus, for every finite rank tensor $\alpha\in J_t^0$, the function $f_\alpha$
is holomorphic and $\hat{f}_\alpha(e)=\alpha$, and so
Proposition \ref{p.Tiso} implies that
$f_\alpha\in\mathcal{H}_t^2(G_{CM})$.  Hence, the image of $f\mapsto\hat{f}(e)$
is dense in $J^0_t$, which suffices to prove surjectivity.
\end{proof}

The following is an immediate consequence of Lemma \ref{l.h7.3}
and Theorem \ref{t.surj}.

\begin{cor}
\label{c.h7.4}
The vector space,
\[
\mathcal{H}_{t,\rm{fin}}^2(G_{CM})
	:=\left\{  f\in \mathcal{H}_t^2(G_{CM}) :
		\hat{f}(e)\in J_t^0  \text{ is finite rank}\right\}
\]
is a dense subspace of $\mathcal{H}_t^2(G_{CM})$.
\end{cor}

\section{The restriction map}
\label{s.restriction}

In this section, we construct the ``skeleton'' or ``restriction''
map between a class of square integrable holomorphic functions on
$G$ and $\mathcal{H}_t^2(G_{CM})$, and we prove that this map is
an isometric isomorphism.  Before proceeding, we must first define
the appropriate class of holomorphic functions on $G$ we wish to
deal with.

Recall from Definition \ref{d.cyl} that a function
$f:G\rightarrow\mathbb{C}$ is a {\em cylinder function} if
$f=F\circ\pi_P$ for some $P\in\mathrm{Proj}(W)$ and
$F:G_P\rightarrow\mathbb{C}$.  We say that $f$ is a {\em holomorphic
cylinder polynomial} if $F$ is a holomorphic polynomial on $G_P$.
The space of holomorphic cylinder polynomials will be denoted by
$\mathcal{P}$.  Propositions \ref{p.rho} and \ref{p.rho2} imply
that $\mathcal{P}\subset L^p(\nu_t)$ for all $p\in[1,\infty)$, so we may make
the following definition.

\begin{defn}
\label{d.htp}
For $t>0$, let $\mathcal{H}^2_t(G)$ denote the $L^2(\nu_t)$-closure of
$\mathcal{P}$.
\end{defn}

\begin{remark}
\label{r.cyl}
Let $\mathcal{A}$ denote the class of holomorphic cylinder functions on $G$.
As remarked in \cite{DG08-3}, it is natural to expect that
$\mathcal{H}_t^2(G)$ coincides with the closure of $\mathcal{A}\cap L^2(\nu_t)$ in $L^2(\nu_t)$, however, this is currently not known even in much simpler settings.
But in a sense $\mathcal{H}_t^2(G)$ is the appropriate space
to consider, as the polynomials should constitute a dense subset
of the square integrable holomorphic functions, when one can make sense
of polynomials.
\end{remark}

In Section \ref{s.dense2}, we show that the restriction of holomorphic
cylinder polynomials to $G_{CM}$ constitutes a dense subspace of
$\mathcal{H}_t^2(G_{CM})$, and with this result in hand, in
Section \ref{s.res2} we construct the restriction map as a linear
map on $\mathcal{H}^2_t(G)$.

\subsection{Another density theorem}
\label{s.dense2}
Techniques similar to those used in this section were used in \cite{DG08-3},
as well as in
Cecil \cite{Cecil08} to prove an analogous result for path groups
over stratified Lie groups.

\begin{thm}
\label{t.pcm}
For all $t>0$,
\[ \mathcal{P}_{CM} := \{p|_{G_{CM}}:p\in\mathcal{P}\} \]
is a dense subspace of $\mathcal{H}_t^2(G_{CM})$.
\end{thm}

This result is analogous to Theorem 7.1 of \cite{DG08-3}, and as
done in that paper, Theorem \ref{t.pcm} will be proved by showing
that $\mathcal{P}_{CM}$ is dense in yet another dense subspace of
$\mathcal{H}_t^2(G_{CM})$.  In particular, Corollary \ref{c.h7.4}
implies that it suffices to show that any element of
$\mathcal{H}_{t,\rm{fin}}^2(G_{CM})$ may be approximated by elements
of $\mathcal{P}_{CM}$.  However, the fact that in our case $J_t^0$
is defined not using the full Hilbert-Schmidt norm complicates some
limiting arguments that appear in \cite{DG08-3}.

Again we recall Notation \ref{n.pn}:
let $\{\xi_j\}_{j=1}^\infty\subset H_*$ be a complex orthonormal basis of $H$
and let $\{\eta_j\}_{j=1}^\infty=\{(\xi_j,0)\}_{j=1}^\infty$.
Define $P_n \in\operatorname*{Proj}(W)$ by
\[
P_n w = \sum_{j=1}^n
	\langle w,\xi_j\rangle_H \xi_j \text{ for all } w\in W,
\]
and $\pi_n:G\rightarrow G_n = P_nW\times\mathbf{C}$ defined by
$\pi_n(w,c)=(P_nw,c)$.

We will show that for all $f\in\mathcal{H}_{t,\rm{fin}}(G_{CM})$,
$f\circ\pi_n\in\mathcal{P}$ and $f\circ\pi_n|_{G_{CM}}\rightarrow f$ in
$\mathcal{H}_t^2(G_{CM})$.  The proof of this statement is complicated by
the fact that,
for general $\omega$ and $P\in\operatorname*{Proj}(W)$,  $\pi_P
:G\rightarrow G_P\subset G_{CM}$ is not a group homomorphism.
In fact, for $g=(w,c)$ and $g'=(w',c')$,
\[ \pi_P(gg') - \pi_Pg \cdot \pi_Pg'
	= \Gamma_P(w,w') \]
where
\begin{equation}
\label{e.GammaP}
\Gamma_P(w,w')
	:= \frac{1}{2}(0,\omega(w, w') - \omega(Pw,Pw'))
	= \frac{1}{2}\left([g,g']-[\pi_Pg,\pi_Pg']\right).
\end{equation}
So unless $\omega$ is ``supported'' on the range of $P$,  $\pi_P$
is not a group homomorphism.  Note that the case where $\omega$ {\em is}
supported on a finite-dimensional space is exactly the trivial case where $L$
is ``finitely many steps from being elliptic,'' and the proof of several of
the other results included here would be greatly simplified.

The proof of the following proposition is similar to Proposition \ref{p.lder}
and is left to the reader.

\begin{prop}
\label{p.hder}
For any $P\in\mathrm{Proj}(W)$,
$g=(w,c)\in G$, $h_i=(A_i,a_i)\in\mathfrak{g}$, and
$f:G\rightarrow\mathbb{C}$ a smooth function,
\begin{equation}
\label{e.hder}
\tilde{h}_n\cdots\tilde{h}_1(f\circ\pi_P)(g)
	= \sum_{k=\lceil n/2\rceil}^n f^{(k)}(\pi_P g)
		\sum_{\theta\in\Lambda_{n-k}^n}
		(h_n,\ldots,h_1)^{\otimes\theta}_P(g),
\end{equation}
where,
for $\theta = \{ \{i_1,i_2\},\ldots,\{i_{2k-1},
i_{2k}\},\{i_{2k+1}\},\ldots,\{i_n\}\}\in\Lambda^n_k$
a partition of $\{1,\cdots,n\}$ as defined in Notation \ref{n.part},
\[
(h_n,\ldots,h_1)_P^{\otimes\theta}(g)
	:= [h_{i_1},h_{i_2}]\otimes \cdots\otimes
		[h_{i_{2k-1}},h_{i_{2k}}]
	\otimes h_{i_{2k+1}}^P(g)\otimes\cdots \otimes h_{i_n}^P(g),
\]
with
\[ h^P(g) := \left(PA, a + \frac{1}{2}\omega(w,A)\right).  \]
\end{prop}

Again as we did for Proposition \ref{p.lder}, let us write out
(\ref{e.hder}) for the first few $n$:
\begin{align*}
\tilde{h}_1(f\circ \pi)(g)
	&= f'(\pi g) h_1^P(g) \\
\tilde{h}_2\tilde{h}_1(f\circ \pi)(g)
	&= f''(\pi g) \left( h_2^P(g)\otimes h_1^P(g)\right)
		+ f'(\pi g) [h_2,h_1] \\
\tilde{h}_3\tilde{h}_2\tilde{h}_1(f\circ \pi)(g)
	&=f'''(\pi g) \left(h_3^P(g) \otimes h_2^P(g)\otimes
		h_1^P(g)\right) \\
	&\hspace{-2ex}+ \frac{1}{2}f''(\pi g) \left(
		[h_3,h_2]\otimes h_1^P(g)
		+  [h_3,h_1]\otimes h_2^P(g) + [h_2,h_1]\otimes h_3^P(g)\right)
\end{align*}
In particular, when $g=e$ and $h_i=(A_i,0)$, we have
$h_i^P(e)=(PA_i,0)=\pi h_i$, and the above formulae become
\begin{align}
\label{e.hder1}
\tilde{h}_1(f\circ \pi)(e) &= f'(e)\pi h_1 \\
\label{e.hder2}
\tilde{h}_2\tilde{h}_1(f\circ \pi)(e)
	&= f''(e)(\pi h_2\otimes \pi h_1)+ f'(e)\frac{1}{2}[h_2,h_1] \\
\label{e.hder3}
\tilde{h}_3\tilde{h}_2\tilde{h}_1(f\circ \pi)(e)
	&= f'''(e)(\pi h_3 \otimes \pi h_2 \otimes \pi h_1) \\
	&\notag \hspace{-2ex}+ \frac{1}{2}
		f''(e)\big([h_3,h_2]\otimes \pi h_1 + \pi h_2\otimes [h_3,h_1]
		+ \pi h_3\otimes [h_2,h_1]\big).
\end{align}

Now using Propositions \ref{p.lder} and \ref{p.hder} we may prove the
following.

\begin{prop}
\label{p.h7.11}
Fix $k\in\mathbb{N}$ and suppose that $f\in\mathcal{H}(G_{CM})$
satisfies $\|\hat{f}_k(e)\|_k<\infty$.  Then
\[
\lim_{n\rightarrow\infty}
	\left\| \hat{f}_k(e) - \left(\widehat{f\circ\pi_n}\right)_k(e)\right\|_k
	= 0.
\]
\end{prop}

\begin{proof}
We will write out the first few cases for small $k$ before proving the
convergence for arbitrary $k$.  Consider first $k=1$.  Then Propositions \ref{p.lder} and
\ref{p.hder}, (more particularly, equations (\ref{e.lder1}) and
(\ref{e.hder1})) imply that
\begin{align*}
\|\hat{f}_1(e) - (\widehat{f\circ\pi})_1(e)\|_1^2
	&= \sum_{j=1}^\infty \left|\tilde{\eta}_jf(e) -
		\tilde{\eta}_j(f\circ\pi)(e)\right|^2 \\
	&= \sum_{j=1}^\infty \left|f'(e)\eta_j-
		f'(e) \pi \eta_j\right|^2
	= \sum_{j=n+1}^\infty \left|f'(e)\eta_j\right|^2
	\rightarrow 0
\end{align*}
as $n\rightarrow\infty$, since by hypothesis
\[\|\hat{f}_1(e)\|_1^2
	= \sum_{j=1}^\infty \left|\tilde{\eta}_jf(e)\right|^2
	= \sum_{j=1}^\infty \left|f'(e)\eta_j\right|^2 <\infty.
\]
Now, for $k=2$, equations (\ref{e.lder2}) and (\ref{e.hder2})) give
\begin{align*}
\|\hat{f}_2(e) &- (\widehat{f\circ\pi})_2(e)\|_2^2
	= \sum_{j_1,j_2=1}^\infty \left|\tilde{\eta}_{j_2}\tilde{\eta}_{j_1}f(e) -
		\tilde{\eta}_{j_2}\tilde{\eta}_{j_1}(f\circ\pi)(e)\right|^2 \\
	&= \sum_{j_1,j_2=1}^\infty \bigg|\left\{f''(e)
		(\eta_{j_1}\otimes \eta_{j_2}) +
		\frac{1}{2}f'(e)[\eta_{j_1},\eta_{j_2}] \right\}\\
	&\qquad \qquad - \left\{f''(e)( \pi \eta_{j_1}\otimes \pi \eta_{j_2}) +
		\frac{1}{2}f'(e)[\eta_{j_1},\eta_{j_2}]\right\} \bigg|^2 \\
	&= \sum_{j_1,j_2=1}^\infty \left|f''(e)
		(\eta_{j_1}\otimes \eta_{j_2} - \pi\eta_{j_1}\otimes \pi\eta_{j_2})
		\right|^2 \\
	&\le \sum_{j_1=1}^\infty\sum_{j_2=n+1}^\infty \left|f''(e)
		(\eta_ {j_1}\otimes \eta_{j_2}) +
		\frac{1}{2}f'(e)[\eta_{j_1},\eta_{j_2}]\right|^2 \\
	&\qquad + \sum_{j_1=n+1}^\infty\sum_{j_2=1}^\infty \left|f''(e)
		(\eta_{j_1}\otimes \eta_{j_2}) +
		\frac{1}{2}f'(e)[\eta_{j_1},\eta_{j_2}]\right|^2 \\
	&\qquad + \frac{1}{2}\sum_{j_1=1}^\infty\sum_{j_2=n+1}^\infty
		|f'(e)[\eta_{j_1},\eta_{j_2}]|^2
		+ \frac{1}{2}\sum_{j_1=n+1}^\infty\sum_{j_2=1}^\infty
		|f'(e)[\eta_{j_1},\eta_{j_2}]|^2
	\rightarrow 0,
\end{align*}
as $n\rightarrow\infty$, since
\[
\|\hat{f}_2(e)\|_2^2 = \sum_{j_1,j_2=1}^\infty \left|f''(e)
		(\eta_{j_1}\otimes \eta_{j_2}) +
		\frac{1}{2}f'(e)[\eta_{j_1},\eta_{j_2}]\right|^2 <\infty,
\]
by hypothesis, and
\[ \sum_{j_1,j_2=1}^\infty |f'(e)[\eta_{j_1},\eta_{j_2}]|^2
	\le |f'(e)|^2 \sum_{j_1,j_2=1}^\infty \|\omega(\xi_{j_1},
		\xi_{j_2})\|_\mathbf{C}^2
	= |f'(e)|^2 \|\omega\|_{HS}^2 <\infty, \]
by Proposition \ref{t.h5.2} which states that $f'(e)$ is a bounded operator
on $G_{CM}$ and Proposition \ref{p.norm} which implies that $\omega$ is
Hilbert-Schmidt.

For $k=3$, equations (\ref{e.lder3}) and (\ref{e.hder3})) give
\begin{align*}
\|&\hat{f}_3(e) - (\widehat{f\circ\pi})_3(e)\|_2^2
	= \sum_{j_1,j_2,j_3=1}^\infty
		\left|\tilde{\eta}_{j_3}\tilde{\eta}_{j_2}\tilde{\eta}_{j_1}f(e)
		- \tilde{\eta}_{j_3}\tilde{\eta}_{j_2}\tilde{\eta}_{j_1}
		(f\circ\pi)(e)\right|^2 \\
	&= \sum_{j_1,j_2,j_3=1}^\infty
		\bigg|f'''(e)(\eta_{j_3}\otimes \eta_{j_2}\otimes \eta_{j_1}
		- \pi \eta_{j_3}\otimes \pi \eta_{j_2}\otimes \pi \eta_{j_1}) \\
	&\qquad + \frac{1}{2}f''(e)([\eta_{j_3},\eta_{j_2}]\otimes \eta_{j_1}
		+ [\eta_{j_3},\eta_{j_1}] \otimes \eta_{j_2}
		+ [\eta_{j_2},\eta_{j_1}]\otimes \eta_{j_3} \\
	&\qquad\qquad - [\eta_{j_3},\eta_{j_2}]\otimes \pi \eta_{j_1} -
		[\eta_{j_3},\eta_{j_1}] \otimes \pi \eta_{j_2}-
		[\eta_{j_2},\eta_{j_1}] \otimes \pi \eta_{j_3})\bigg|^2 \\
	&\le \sum_{\ell=1}^3 \sum_{j_\ell=n+1}^\infty
		\sum_{\tiny\begin{array}{cc}j_i=1\\i\ne\ell\end{array}}^\infty
		\bigg|f'''(e)(\eta_{j_3}\otimes \eta_{j_2}\otimes \eta_{j_1}) \\
	&\qquad + \frac{1}{2}f''(e)([\eta_{j_3},\eta_{j_2}]\otimes \eta_{j_1}
		+ [\eta_{j_3},\eta_{j_1}] \otimes \eta_{j_2}
		+ [\eta_{j_2},\eta_{j_1}]\otimes \eta_{j_3}) \bigg|^2
	\rightarrow 0
\end{align*}
as $n\rightarrow\infty$, since
\begin{align*}
\|\hat{f}_3(e)\|_2^2
	&= \sum_{j_1,j_2,j_3=1}^\infty
			\bigg|f'''(e)(\eta_{j_3}\otimes \eta_{j_2}\otimes \eta_{j_1}) \\
	&\qquad + \frac{1}{2}f''(e)([\eta_{j_3},\eta_{j_2}]\otimes \eta_{j_1}
		+ [\eta_{j_3},\eta_{j_1}] \otimes \eta_{j_2}
		+ [\eta_{j_2},\eta_{j_1}]\otimes \eta_{j_3}) \bigg|^2
	< \infty,
\end{align*}
again by hypothesis.

More generally, using equations (\ref{e.lder}) and (\ref{e.hder}) with $g=e$ and
$\eta_j = (\xi_j,0)$ for $k$ odd shows that
\begin{align*}
\|&\hat{f}_k(e) - (\widehat{f\circ\pi})_k(e)\|_k^2
	\le \sum_{\ell=1}^k \sum_{j_\ell=n+1}^\infty
		\sum_{\tiny\begin{array}{cc}j_i=1\\i\ne\ell\end{array}}^\infty
		\left|\langle \hat{f}_k(e),\eta_{j_k}\otimes\cdots\otimes
		\eta_{j_1}\rangle\right|^2
	\rightarrow 0
\end{align*}
as $n\rightarrow\infty$.  Similarly, for $k$ even,
\begin{align*}
\|\hat{f}_k(e) - (\widehat{f\circ\pi})_k(e)\|_k^2
	&\le \sum_{\ell=1}^k \sum_{j_\ell=n+1}^\infty
		\sum_{\tiny\begin{array}{cc}j_i=1\\i\ne\ell\end{array}}^\infty
		\bigg\{\left|\langle \hat{f}_k(e),\eta_{j_k}\otimes\cdots\otimes
		\eta_{j_1}\rangle\right|^2 \\
	&\qquad + \frac{1}{2} \sum_{\theta\in\Lambda_{k/2}^k}
		\left|f^{(k/2)}(e) (\eta_{j_k},\ldots,
		\eta_{j_1})^{\otimes \theta} \right|^2\bigg\}
	\rightarrow 0,
\end{align*}
as $n\rightarrow\infty$,
since for $\theta= \{\{i_1,i_2\},\ldots,\{i_{k-1},i_k\}\}
\in\Lambda_{k/2}^k$, we have
\[ (\eta_{j_k},\ldots,\eta_{j_1})^{\otimes \theta}
	= [\eta_{j_{i_1}},\eta_{j_{i_2}}]\otimes\cdots\otimes
		[\eta_{j_{i_{k-1}}},\eta_{j_{i_k}}],
\]
which implies that
\begin{align*}
\sum_{j_1,\ldots,j_k=1}^\infty &
		\left|f^{(k/2)}(e) (\eta_{j_k},\ldots,\eta_{j_1})^{\otimes \theta}
		\right|^2 \\
	&\le \left|f^{(k/2)}(e)\right|^2 \sum_{j_1,\ldots,j_k=1}^\infty
		\|[\eta_{j_{i_1}},\eta_{j_{i_2}}]\|_\mathbf{C}^2\cdots
		\|[\eta_{j_{i_{k-1}}},\eta_{j_{i_k}}]\|_\mathbf{C}^2  \\
	&= \left|f^{(k/2)}(e)\right|^2 \|\omega\|_{HS}^{n}
	<\infty,
\end{align*}
again by Propositions \ref{t.h5.2} and \ref{p.norm}.
\end{proof}

The following proposition completes the proof of Theorem \ref{t.pcm}.

\begin{prop}
\label{p.h7.12}
If $f\in\mathcal{H}_{t,\rm{fin}}^2(G_{CM})$ as defined in Corollary
\ref{c.h7.4},
then $f\circ\pi_n \in \mathcal{P}$ for all $n\in\mathbb{N}$ and
$f\circ\pi_n|_{G_{CM}}\rightarrow f$ in $\mathcal{H}_t^2(G_{CM})$.
\end{prop}

\begin{proof}
Suppose $m\in\mathbb{N}$ is chosen so that $\hat{f}_k(e) = 0$
if $k>m$. Comparing equations (\ref{e.lder}) and (\ref{e.hder}), one may
determine that, for $h_1,\ldots,h_k\in\mathfrak{g}_{CM}$,
\begin{equation}
\label{e.kap}
\left\langle\left(\widehat{f\circ\pi_n}\right)(e),
		h_k\otimes\dots\otimes h_1\right\rangle
	= \left\langle \hat{f}(e),\kappa^n_k(h_k,\ldots,h_1)\right\rangle,
\end{equation}
where $\kappa_k^n$ is defined as follows: for $h_i=(A_i,a_i)$,
\[ \kappa_k^n(h_k,\ldots,h_1)
	:= \sum_{j=\lfloor k/2\rfloor}^k \sum_{\theta\in\Lambda_{k-j}^k}
		\Gamma_{P_n}^{\otimes\theta}(h_k,\cdots,h_1), \]
where, for $\theta =\{ \{i_1,i_2\},\ldots,\{i_{2\ell-1},
i_{2\ell}\},\{i_{2\ell+1}\},\ldots,\{i_k\}\}\in\Lambda^k_\ell$,
\[ \Gamma_{P_n}^{\otimes\theta} (h_k,\ldots,h_1)
	:= \Gamma_{P_n}(A_{i_1},A_{i_2})\otimes\cdots\otimes
	\Gamma_{P_n}(A_{i_{2\ell-1}},A_{i_{2\ell}})\otimes \pi h_{i_{2\ell+1}}\otimes
	\cdots\otimes\pi h_{i_k}, \]
and $\Gamma_P(A_i,A_j) = \frac{1}{2}([h_i,h_j]-[\pi h_i,\pi h_j])$
as in equation (\ref{e.GammaP}).  Alternatively, one may consult
Section 7.2 of
\cite{DG08-3} for a direct derivation of $\kappa_k^n$ and equation (\ref{e.kap})
(in this reference, our $\kappa_k^n(h_k,\ldots,h_1)$ is just $\kappa_k(e)$).

By definition,
$\kappa^n_k(h_k,\ldots,h_1)\in\bigoplus_{j=\lceil k/2\rceil}^k
\mathfrak{g}_{CM}^{\otimes j}$ and so (\ref{e.kap}) implies that
$\left\langle (\widehat{f\circ\pi_n})(e),h_k\otimes\dots\otimes
h_1\right\rangle =0$ when $k\ge2m+2$.  Therefore, $f\circ\pi_n$ restricted to
$G_n=P_nH\times\mathbf{C}$ is a holomorphic polynomial, and, since
$f\circ\pi_n=(f\circ\pi_n)|_{G_n}\circ\pi_n$,  it follows
that $f\circ\pi_n\in\mathcal{P}$.

Moreover,
\[
\lim_{n\rightarrow\infty}
		\left\| \hat{f}(e)-\left(\widehat {f\circ\pi_n}\right)(e)
		\right\|_t^2
	= \lim_{n\rightarrow\infty}\sum_{k=0}
		^{2m+2}\frac{t^k}{k!}\left\| \hat{f}_k(e)
		- \left(\widehat{f\circ\pi_n}\right)_k(e)\right\|_k^2 = 0,
\]
since Proposition \ref{p.h7.11} implies that
$\lim_{n\rightarrow\infty}
	\left\| \hat{f}_k(e)-\left(\widehat{f\circ\pi_n}\right)_k(e)\right\|_k=0$
for each $k$.  Thus, by Proposition \ref{p.Tiso},
\[ \lim_{n\rightarrow\infty}
	 \left\| f-f\circ\pi_n\right\|_{\mathcal{H}_t^2(G_{CM})}
	=\lim_{n\rightarrow\infty}
		\left\| \hat{f}(e)-\left(\widehat {f\circ\pi_n}\right)(e)
		\right\|_t=0.  \]
\end{proof}

\subsection{Construction and proof of restriction isomorphism}
\label{s.res2}

Before we construct the restriction map, we require some preliminary
estimates.  Again, we let
$\{\eta_j\}_{j=1}^\infty=\{(\xi_j,0)\}_{j=1}^\infty\subset
H_*\times\{0\}$, $\{P_n\}_{n=1}^\infty\subset\mathrm{Proj}(W)$, and
$\pi_n:G\rightarrow G_n$ be as in Notation \ref{n.pn}.  Also, for
$f:G\rightarrow\mathbb{C}$ or $f:G_{CM}\rightarrow\mathbb{C}$, let
\[ \|f\|^2_{L^2(\nu_t^n)} := \|f|_{G_n}\|^2_{L^2(\nu_t^n)} = \mathbb{E}|f(g_t^n)|^2, \]
where $\{g_t^n\}_{t\ge0}\subset G_n\subset G_{CM}\subset G$
is a Brownian motion on $G_n$ as in Proposition \ref{p.bmapprox}.

First we show that these norms are increasing in $n$ (for sufficiently large
$n$).  A similar result was proved in \cite[Lemma 4.1]{Gordina00-2}.

\begin{lem}
\label{l.inc}
Suppose $f:G\rightarrow\mathbb{C}$ is a continuous function such that
$f|_{G_n}\in\mathcal{H}(G_n)$ for all $n\in\mathbb{N}$.  Then
$\|f\|_{L^2(\nu_t^n)}\le \|f\|_{L^2(\nu_t^{n+1})}$ for all  large enough $n\in\mathbb{N}$.
\end{lem}
\begin{proof}
For each $n\in\mathbb{N}$, let $D_n=D_{P_n}^k$ where $D_{P_n}^k$ is as defined
in Notation \ref{n.ffhat}.  By the Taylor isomorphism for subelliptic
heat kernels on finite dimension Lie groups stated in Theorem \ref{t.fdt},
\[ \|f\|_{L^2(\nu_t^n)} = \|\hat{f}(e)\|_{n,t}, \]
where we recall that
\[ \|\hat{f}(e)\|_{n,t}^2
    = \sum_{k=0}^\infty \frac{t^k}{k!} \|(D_n^kf(e)\|^2_{n,k}, \]
for all $n$ sufficiently large that $[P_nW,P_nW]=\mathbf{C}$.
Observing that, for each such $n\in\mathbb{N}$ and $k\in\{0\}\cup\mathbb{N}$,
\begin{align*}
\|(D_n^kf)(e)\|^2_{n,k}
    &= \sum_{j_1,\dots,j_k=1}^n |\langle
        (D_n^kf)(e), \eta_{j_1}\otimes\cdots\otimes\eta_{j_k}\rangle|^2 \\
    &= \sum_{j_1,\dots,j_k=1}^n
        |\tilde{\eta}_{j_1}\cdots\tilde{\eta}_{j_k}f(e)|^2
    \le \sum_{j_1,\dots,j_k=1}^{n+1}
        |\tilde{\eta}_{j_1}\cdots\tilde{\eta}_{j_k}f(e)|^2 \\
    &= \sum_{j_1,\dots,j_k=1}^{n+1} |\langle
        (D_{n+1}^kf)(e), \eta_{j_1}\otimes\cdots\otimes\eta_{j_k}\rangle|^2
    = \|(D_{n+1}^kf)(e)\|^2_{n+1,k},
\end{align*}
completes the proof.
\end{proof}

\begin{lem}
\label{l.2}
For any continuous function $f:G\rightarrow\mathbb{C}$
such that $f|_{G_{CM}}\in\mathcal{H}(G_{CM})$,
\[ \|f\|_{L^2(\nu_t)} \le \|f|_{G_{CM}}\|_{\mathcal{H}_t^2(G_{CM})}. \]
\end{lem}

\begin{proof}
First, note that, if
$\{P_n\}_{n=1}^\infty\subset\mathrm{Proj}(W)$ such that $P_n|_H\uparrow I_H$,
then Proposition \ref{p.bmapprox} implies that (passing to a subsequence if
necessary) $g_t^n\rightarrow g_t$ almost surely. Thus,
\[ \|f\|_{L^2(\nu_t)} \le \sup_n \|f\|_{L^2(\nu_t^n)}
    \le \|f|_{G_{CM}}\|_{\mathcal{H}_t^2(G_{CM})}, \]
where the first inequality holds by Fatou's lemma and the second by the
definition of $\|\cdot\|_{\mathcal{H}_t^2(G_{CM})}$.
\end{proof}

\begin{remark}
Of course this lemma holds for any $p\in[1,\infty)$, for
$\mathcal{H}^p_t(G)$ defined analogously to $\mathcal{H}^2_t(G)$ in
Definition \ref{d.htp}.
\end{remark}

\begin{cor}
\label{c.siso}
Let $\delta>0$ be
as in Proposition \ref{p.rho}, and suppose that
$f:G\rightarrow\mathbb{C}$ is a continuous function such that, for some
$\varepsilon\in(0,\delta)$,
\[ |f(g)|\le C e^{\varepsilon\|g\|_\mathfrak{g}^2/2t}, \]
for all $g\in G$.  Then
\[ \|f\|_{L^2(\nu_t^n)}\uparrow \|f\|_{L^2(\nu_t)}. \]
(In particular, this implies that
$\|f\|_{L^2(\nu_t^P)} \le \|f\|_{L^2(\nu_t)}$
for any $P\in\rm Proj(W)$.)
Also, if $f|_{G_{CM}}\in\mathcal{H}(G_{CM})$, then
\begin{equation}
\label{e.siso}
\|f\|_{L^2(\nu_t)} = \|f|_{G_{CM}}\|_{\mathcal{H}_t^2(G_{CM})}.
\end{equation}
\end{cor}
\begin{proof}
First, Lemma \ref{l.inc} implies that
$\{\|f\|_{L^2(\nu_t^n)}\}_{n=1}^\infty$ is an increasing sequence.
Proposition \ref{p.rho2} implies that $f\in L^2(\nu_t)$, and
taking $h=e$ in equation (\ref{e.h1}) or equation (\ref{e.h2}) shows
that the sequence must be increasing to $\|f\|_{L^2(\nu_t)}$.
This combined with Lemma \ref{l.2} gives (\ref{e.siso}).
\end{proof}

\begin{lem}
\label{l.dP}
Suppose $f:G\rightarrow\mathbb{C}$ is a continuous function such that
$f|_{G_n}\in\mathcal{H}L^2(\nu_t^n)$ for all $n\in\mathbb{N}$.  Then, for all
$g\in G_{CM}$,
\[ |f(g)| \le \|f\|_{L^2(\nu_t)}e^{d_h(e,g)^2/2t}. \]
\end{lem}

\begin{proof}
Let $g=(w,c)\in G_m$, and consider an arbitrary horizontal path
$\sigma:[0,1]\rightarrow G_{CM}$ such that $\sigma(0)=e$ and
$\sigma(1)=g$.  Recall that, by Remark \ref{r.horiz}, $\sigma$ must have the
form
\[ \sigma(t) = \left(A(t),\frac{1}{2}\int_0^t \omega(A(s),\dot{A}(s))\,ds
	\right). \]
For $n\ge m$, consider the ``projected'' horizontal paths
$\sigma_n:[0,1]\rightarrow G_n$ given by
\[ \sigma_n(t)
	= (A_n(t),a_n(t))
	:= \left(P_nA(t),\frac{1}{2}\int_0^t \omega(P_nA(s),P_n\dot{A}(s))
		\,ds\right). \]
Note that $A_n(1)=P_nA(1)=P_nw=w$, and let
\[ \varepsilon_n
	:= c- a_n(1)
	= c-\frac{1}{2}\int_0^1 \omega(P_nA(s),P_n\dot{A}(s))\,ds\in\mathbf{C}. \]
Then, for $d_n$ the horizontal distance in $G_n$,
\begin{align}
\notag
d_n(e,g) &= d_n(e,(w,c))
	= d_n(e,(w,a_n(1)+\varepsilon_n))
	= d_n(e,(w,a_n(1))\cdot(0,\varepsilon_n)) \\
	&\notag
	\le d_n(e,(w,a_n(1))) + d_n(e,(0,\varepsilon_n)) \\
	&\label{e.bb}
\le \ell(\sigma_n) + C\sqrt{\|\varepsilon_n\|_\mathbf{C}},
\end{align}
where the first inequality holds by (\ref{e.12.2}) and the second inequality
holds by (\ref{e.12.6}), with constant $C=C(N,\omega)$. Note that \eqref{e.12.5} technically gives only a bound for $d_{h}$ on $G_{CM}$; however, it is clear from the proof of this bound that one may find a constant $C$ so that \eqref{e.12.6} holds for all sufficiently large $n$ with the constant $C$ not depending on $n$.

Now consider a continuous function $f:G\rightarrow\mathbb{C}$ such that
$f|_{G_n}\in\mathcal{H}L^2(\nu_t^n)$ for all $n\in\mathbb{N}$.
For $n\ge m$ , $g\in G_m\subset G_n$.  Then, for $n$ sufficiently large that
$[P_nW,P_nW]=\mathbf{C}$,
Theorem \ref{t.fdt} (in particular (\ref{e.ptwise})), Corollary \ref{c.siso},
and (\ref{e.bb}) imply that
\begin{equation}
\label{e.aa}
|f(g)|\le \|f\|_{L^2(\nu_t^n)}e^{d_n(e,g)^2/2t}
	\le \|f\|_{L^2(\nu_t)}
		e^{(\ell(\sigma_n)+C\sqrt{\|\varepsilon_n\|_\mathbf{C}})^2/2t}.
\end{equation}
One may then show via dominated convergence that
\[ \lim_{n\rightarrow\infty} \ell(\sigma_n)
	= \lim_{n\rightarrow\infty} \int_0^1 \|P_n\dot{A}(s)\|\,ds
	= \int_0^1 \|\dot{A}(s)\|\,ds
	= \ell(\sigma), \]
and that
\[ \lim_{n\rightarrow\infty} \|\varepsilon_n\|_\mathbf{C}
	= \lim_{n\rightarrow\infty}
		\left\|\frac{1}{2}\int_0^1 \omega(A(s),\dot{A}(s)) -
		\omega(P_nA(s),P_n\dot{A}(s))\,ds \right\|_\mathbf{C}
	= 0.
\]
Thus, passing to the limit in (\ref{e.aa}) as $n\rightarrow\infty$ gives
\[ |f(g)|	\le \|f\|_{L^2(\nu_t)}e^{\ell(\sigma)^2/2t}, \]
and taking the infimum over all horizontal paths $\sigma$ such that
$\sigma(0)=e$ and $\sigma(1)=g$ completes the proof for all $g\in\cup_P
G_P$.  Since both sides of the inequality are continuous in $g\in
G_{CM}$
and $\cup_P G_P$ is dense in $G_{CM}$ by Proposition
\ref{p.equivtop}, this is sufficient to prove the bound for all $g\in
G_{CM}$.
\end{proof}

\begin{notation}
For $g\in G_{CM}$, define the linear map $R_g:\mathcal{P}
\rightarrow\mathbb{C}$  by
\[ R_g f := f(g). \]
\end{notation}

\begin{prop}
\label{p.Rg}
For all $g\in G_{CM}$, $R_g$ can be extended uniquely to a continuous linear
functional on all of $\mathcal{H}_t^2(G)$ satisfying
\begin{equation}
\label{e.Rg}
|R_gf| \le \|f\|_{L^2(\nu_t)}e^{d_h(e,g)^2/2t}.
\end{equation}
\end{prop}

\begin{proof}
Lemma \ref{l.dP} implies that (\ref{e.Rg}) holds for $f\in \mathcal{P}$
and $g\in G_{CM}$.  Thus, $\|R_g\|_{op}\le e^{d_h(e,g)^2/2t}$ as an operator
on $\mathcal{P}\subset L^2(\nu_t)$, and $R_g$ is continuous and defined
on a dense subset of $\mathcal{H}_t^2(G)$.  Thus, there exists a unique
extension of $R_g$ to $\mathcal{H}_t^2(G)$ (still denoted by $R_g$) so that
(\ref{e.Rg}) is satisfied for all $f\in\mathcal{H}_t^2(G)$.
To define $R_g$ for an arbitrary $f\in\mathcal{H}_t^2(G)$, let
$\{f_j\}_{j=1}^\infty\subset\mathcal{P}$ such that
$f_j\rightarrow f$ in $L^2(\nu_t)$
and define $R_gf := \lim_{j\rightarrow\infty} R_gf_j$.
\end{proof}

\begin{remark}
The estimate in (\ref{e.Rg}) implies that, if $f_j\rightarrow f$ in
$L^2(\nu_t)$, then, for any $g\in G_{CM}$,
$R_gf_j\rightarrow R_gf$ and the convergence is locally uniform.
\end{remark}

\begin{thm}
\label{t.Rlin}
There exists a linear map $R:\mathcal{H}_t^2(G)\rightarrow\mathcal{H}(G_{CM})$
with the following properties:
\begin{enumerate}
\item For any $f\in\mathcal{P}$, $Rf=f|_{G_{CM}}$.
\item For $g\in G_{CM}$, $|(Rf)(g)|\le \|f\|_{L^2(\nu_t)}e^{d_h^2(e,g)/2t}$.
\end{enumerate}
\end{thm}

\begin{proof}
Given $f\in\mathcal{H}_t^2(G)$, we define $Rf$ by $(Rf)(g):=R_gf$ for all $g\in
G_{CM}$.  Items (1) and (2) are satisfied by definition of $R_g$ and Proposition
\ref{p.Rg}.

To see that $Rf\in\mathcal{H}(G_{CM})$, first consider $f\in\mathcal{P}$.
Then $f=F\circ\pi_P$ for some $P\in\mathrm{Proj}(W)$ and
polynomial $F\in\mathcal{H}(G_P)$.
By Proposition \ref{l.h5.4}, $h\mapsto f(g\cdot e^h)$ is
Frech\'et differentiable at $h=0$
and this derivative is continuous with respect to $g$.

For general $f\in\mathcal{H}_t^2(G)$, fix $g\in G_{CM}$ and choose
$\{f_j\}_{j=1}^\infty\subset\mathcal{P}$ such that
$f_j\rightarrow f$ in $L^2(\nu_t)$.  Then
\[ |(Rf_j)(g)-(Rf)(g)| = |R_g(f_j-f)|
    \le \|f_j-f\|_{L^2(\nu_t)}e^{d_h^2(e,g)/2t}, \]
and so $Rf$ is the pointwise limit of
$Rf_j=f_j|_{G_{CM}}\in\mathcal{H}(G_{CM})$ with the limit being
uniform over any bounded subset of $g$'s contained in $G_{CM}$. By
Theorem 3.18.1 of \cite{HP74},
this is sufficient to imply that $Rf\in\mathcal{H}(G_{CM})$.
\end{proof}

\begin{thm}
\label{t.Rsurj}
The map $R:\mathcal{H}_t^2(G)\rightarrow\mathcal{H}_t^2(G_{CM})$ is unitary.
\end{thm}
\begin{proof}
Given $f\in\mathcal{P}$,
Corollary \ref{c.siso} implies that $\|Rf\|_{\mathcal{H}_t^2(G_{CM})} =
\|f\|_{L^2(\nu_t)}$.  Therefore, $R|_\mathcal{P}$ extends to an isometry,
still denoted by $R$, from $\mathcal{H}_t^2(G)$ to $\mathcal{H}_t^2(G_{CM})$
such that $R(\mathcal{P})=\mathcal{P}_{CM}$.  Since $R$ is isometric and
$\mathcal{P}_{CM}$ is dense in $\mathcal{H}_t^2(G_{CM})$ by Theorem
\ref{t.pcm}, it follows that $R$ is surjective.
\end{proof}

\begin{cor}
\label{c.dd}
Suppose $f:G\rightarrow\mathbb{C}$ is a continuous function such that
$f|_{G_{CM}}\in\mathcal{H}_t^2(G_{CM})$.  Then $f\in\mathcal{H}_t^2(G)$ and
$\|f\|_{L^2(\nu_t)}=\|f|_{G_{CM}}\|_{\mathcal{H}_t^2(G_{CM})}$.
\end{cor}

\begin{proof}
By Theorem \ref{t.Rsurj}, there exists $u\in\mathcal{H}_t^2(G)$ such that
$Ru=f|_{G_{CM}}$.  Let $p_n\in\mathcal{P}$ be chosen so that $p_n\rightarrow
u$ in $L^2(\nu_t)$.  Then $p_n|_{G_{CM}}=Rp_n\rightarrow Ru=f|_{G_{CM}}$ in
$\mathcal{H}_t^2(G_{CM})$, and, by Lemma \ref{l.2},
\[ \|f - p_n\|_{L^2(\nu_t)}
	\le \|(f-p_n)|_{G_{CM}}\|_{\mathcal{H}_t^2(G_{CM})}. \]
Thus, $p_n\rightarrow f$ in $L^2(\nu_t)$, and since $p_n\rightarrow u$ in
$L^2(\nu_t)$ also, it must be that $f=u\in\mathcal{H}_t^2(G)$.
\end{proof}

Corollary \ref{c.dd} along with Corollary \ref{c.siso} immediately give
the following.  In particular, this result states that, under the
assumptions of Corollary \ref{c.siso},
$f\in \mathcal{H}_t^2(G)$.

\begin{cor}
\label{c.cc}
Let $\delta>0$ be as in Proposition
\ref{p.rho}, and suppose that $f:G\rightarrow\mathbb{C}$ is a
continuous function such that $f|_{G_{CM}}\in\mathcal{H}(G_{CM})$ and,
for some $\varepsilon\in(0,\delta)$,
\[ |f(g)|\le C e^{\varepsilon\|g\|_\mathfrak{g}/2t}, \]
for all $g\in G$.  Then $f\in\mathcal{H}_t^2(G)$
and $\|f\|_{L^2(\nu_t)}=\|f|_{G_{CM}}\|_{\mathcal{H}_t^2(G_{CM})}$.
\end{cor}

\bibliographystyle{amsplain}
\providecommand{\bysame}{\leavevmode\hbox to3em{\hrulefill}\thinspace}
\providecommand{\MR}{\relax\ifhmode\unskip\space\fi MR }
\providecommand{\MRhref}[2]{%
  \href{http://www.ams.org/mathscinet-getitem?mr=#1}{#2}
}
\providecommand{\href}[2]{#2}

\end{document}